\documentclass[11pt]{article}

\usepackage[paper=a4paper,left=32mm,right=32mm,top=30mm,bottom=30mm]{geometry}
\usepackage{algorithm} 
\usepackage{algorithmic} 
\usepackage{amsmath} 
\usepackage{amssymb} 
\usepackage{amsthm} 
\usepackage[shortlabels]{enumitem}
    \setlist{nosep}
\usepackage[utf8]{inputenc}
\usepackage[T1]{fontenc}
\usepackage{mathptmx}
\usepackage{pgfplots}
\usepackage{tikz}
    \pgfplotsset{compat=newest}
    \usetikzlibrary{plotmarks}
    \usetikzlibrary{calc}
    \usetikzlibrary{positioning}
    \usetikzlibrary{arrows.meta}
    \usetikzlibrary{matrix}
    \usepgfplotslibrary{patchplots}
\usepackage[labelsep=period]{caption}
\usepackage{xspace}
\usepackage{xcolor} 
\usepackage{cite}
\usepackage{hyperref}
\usepackage[nameinlink]{cleveref} 

\let\oldeqref\eqref
\renewcommand{\eqref}[1]{\oldeqref{eq:#1}}

\mathchardef\ordinarycolon\mathcode`\:
\mathcode`\:=\string"8000
\begingroup \catcode`\:=\active
  \gdef:{\mathrel{\mathop\ordinarycolon}}
\endgroup

\DeclareMathAlphabet{\altmathcal}{OMS}{cmsy}{m}{n}
\renewcommand{\mathcal}{\altmathcal}

\newcommand{\norm}[1]{\Vert #1 \Vert}

\newcommand{\RR}{\mathbb{R}}

\newcommand{\CC}{\mathbb{C}}

\newcommand{\und}{~~ \text{and} ~~}

\newcommand{\spectrum}{\sigma}

\newcommand{\inv}{^{-1}}
\newcommand{\tp}{^{\mathsf{T}}}
\newcommand{\mtp}{^{-\mathsf{T}}}
\newcommand{\hm}{^{\mathsf{H}}}

\newcommand{\hatX}{\hat{X}}
\newcommand{\hatY}{\hat{Y}}

\newcommand{\hinf}{{\mathcal{H}_\infty}}

\renewcommand{\tilde}{\widetilde}
\renewcommand{\hat}{\widehat}
\renewcommand{\Re}{\mathrm{Re}}

\newcommand{\matlab}{MATLAB\textsuperscript{\textregistered}\xspace}

\newcommand{\pH}{\textsf{pH}\xspace}
\newcommand{\portHamiltonian}{port-Ha\-mil\-to\-ni\-an\xspace}
\newcommand{\PortHamiltonian}{Port-Ha\-mil\-to\-ni\-an\xspace}

\pgfplotsset{
  custom y log style/.style={
      yticklabel={
        \pgfkeys{/pgf/fpu=true}
        \pgfmathparse{exp(\tick)}%
        \pgfmathprintnumber[fixed,fixed zerofill, precision=2]{\pgfmathresult}
        \pgfkeys{/pgf/fpu=false}
      }
  }
}

\allowdisplaybreaks

\definecolor{mylinkcolor}{rgb}{0.031,0.2,0.369}
\definecolor{myemphcolor}{rgb}{0.2,0.2,0.2}

\hypersetup{pdfborder={0 0 0.3}}

\title{Structure-Preserving $\hinf$ Control \\ for \PortHamiltonian Systems}

\theoremstyle{plain}
\newtheorem{theorem}{Theorem}

\newtheorem{example}[theorem]{Example}

\newtheorem{proposition}[theorem]{Proposition}
\newtheorem{remark}[theorem]{Remark}

\allowdisplaybreaks 

\begin{document}

\maketitle

 \centerline{\scshape Tobias Breiten}
 \medskip
 {\footnotesize
  \centerline{Institute of Mathematics}
    \centerline{Technische Universit\"at Berlin}
    \centerline{Stra\ss e des 17. Juni 136, 10623 Berlin, Germany}
    \centerline{tobias.breiten@tu-berlin.de}
 }

 \medskip

 \medskip
  \centerline{\scshape Attila Karsai}
 \medskip
 {\footnotesize
  \centerline{Institute of Mathematics}
    \centerline{Technische Universit\"at Berlin}
    \centerline{Stra\ss e des 17. Juni 136, 10623 Berlin, Germany}
    \centerline{karsai@math.tu-berlin.de}
 }

\bigskip

\begin{abstract}
    We study $\hinf$ control design for linear time-invariant \portHamiltonian systems.
    By a modification of the two central algebraic Riccati equations, we ensure that the resulting controller will be \portHamiltonian.
    Using these modified equations, we proceed to show that a corresponding balanced truncation approach preserves \portHamiltonian structure.
    We illustrate the theoretical findings using numerical examples and observe that the chosen representation of the \portHamiltonian system can have an influence on the approximation qualities of the reduced order model.
\end{abstract}


\section{Introduction}
Linear systems are an important tool in mathematical modeling.
Large model classes can be written in a linear form, and many nonlinear systems can be linearized around equlibria to obtain a linear approximation.
Control of such systems is an essential part of many applications.
A common approach is control by interconnection, where the original system, the \emph{plant}, is connected to a second linear system, the \emph{controller}.
Two well-known examples of this technique are linear quadratic Gaussian (LQG) control~\cite{KalB61} and $\hinf$ control~\cite{DoyGKF89}.
The latter is particularly interesting for real-world applications due to the poor robustness properties of LQG control~\cite{Doy78}.

Often, the considered linear systems have additional mathematical properties that can be interpreted physically, such as, for example, passivity or \portHamiltonian structure.
Although the roots of \portHamiltonian (\pH) systems theory date back as far as the late 1950s~\cite{VanJ14}, they continue to be the focus of active research.
For an overview of \portHamiltonian systems, see, e.g.,~\cite{MehU22,Van06,VanJ14} and the references therein.
In the context of control methods, the property that power-conserving interconnections of \pH systems can again be formulated as \portHamiltonian systems~\cite{Van00,Van06} is particularly important.
Unfortunately, classical LQG and $\hinf$ control do not necessarily yield \portHamiltonian controllers, even when the considered plant is \portHamiltonian.
Similarly, they do not preserve other important properties of the plant, which has lead to the development of modified techniques in the past, see for example~\cite{BreMS21,LozJ88} for the LQG setting and~\cite{HadBW93} for the $\hinf$ setting.

Our contribution is the introduction of a structure-preserving controller synthesis method which ensures that the resulting closed loop transfer function stays within a prescribed $\hinf$ margin.
To achieve this goal, we take a similar approach as~\cite{HadBW93} and make use of a result established in~\cite{BerH89}.
\Cref{thm:mhinf} demonstrates how the algebraic Riccati equations used for classical $\hinf$ control need to be altered to ensure that the resulting controller has \portHamiltonian structure.
As in~\cite{BreMS21}, the Hamiltonian of the \pH plant plays an important role in the construction.
To show the closed loop $\hinf$ bound in \Cref{thm:mhinf}, we need to make an additional assumption on the resulting control system, which envolves the existence of solutions to specific Lur'e equations.
In order to overcome this additional assumption, in \Cref{thm:mhinf-v2} we present an extension of the method.
Subsequently, using the algebraic Riccati equations that play a central role in our approach, we are able to develop a model order reduction method which preserves \portHamiltonian structure.
For that, ideas from system balancing and the effort-constrained method from~\cite{PolV12} are used.
Let us point out that other structure-preserving model reduction methods have been proposed.
These include approaches based on Riccati equations~\cite{BreMS21}, spectral factorization~\cite{BreU21}, tangential interpolation~\cite{GugPBV12}, symplectic geometry~\cite{MamZ22}, Krylov methods~\cite{PolV09,PolV10}, and optimization~\cite{SchV20-2}.

The paper is structured as follows.
In \Cref{sec:preliminaries}, we collect the necessary background, introduce the notion of \portHamiltonian systems, and mention connections to Kál\-mán–Ya\-ku\-bo\-vich–Pop\-ov linear matrix inequalities (KYP LMIs).
Further, we shortly discuss the notion of (strictly) positive real systems and recall a definition from~\cite{Wen88} relying on Lur'e equations to circumvent inconsistencies found in the literature.
At the end of the section, we adapt a known statement regarding the power-conserving interconnection of such systems to this setting.
Our main results are stated in \Cref{sec:mainresult}.
After a precise problem formulation and a motivating example, \Cref{thm:mhinf,thm:mhinf-v2} are developed.
Both theorems state a pair of modified algebraic Riccati equations that allow for the construction of a controller that has a \portHamiltonian formulation and ensures an $\hinf$ bound for the closed loop transfer function.
As a consequence of these results, in \Cref{sec:modelreduction} a balancing based model order reduction method is presented.
In \Cref{sec:numexp}, the theoretical results are illustrated using numerical examples.
Finally, \Cref{sec:conclusion} concludes the findings and outlines possible future research objectives.

\subsubsection*{Notation}
We denote the open right and the open left half-plane by $\CC^+$ and $\CC^-$, respectively, and the imaginary axis by $i\RR$.
We denote the real part of a complex number $z \in \CC$ by $\Re(z)$.
The identity matrix with size inferred from the context is denoted by $I$.
Besides, the conjugate transpose of a matrix $A$ is denoted with $A\hm$, its spectrum by $\spectrum(A)$, and its kernel by $\ker(A)$.
Further, the smallest singular value in the economy sized singular value decomposition is termed $\sigma_{\min}(A)$.
In other words, if the matrix $A$ has full rank, then $\sigma_{\min}(A)$ is positive.
For a symmetric matrix $A\in\CC^{n,n}$, the smallest eigenvalue is denoted by $\lambda_{\min}(A)$.
We write $A \succeq 0$ if $x\hm A x \geq 0$ for all $x\in \CC^{n}$, and $A \succ 0$ if $x\hm A x > 0$ for all $x\in\CC^{n}\setminus\{0\}$.
Similarly, for matrices $A$ and $B$ we write $A \succeq B$ if $A-B \succeq 0$.
We call a matrix $A \in \CC^{n,n}$ \emph{stable} if $\spectrum(A) \subseteq \CC^{-} \cup i\RR$ and all purely imaginary eigenvalues of $A$ are semi-simple, and \emph{asymptotically stable} if $\spectrum(A)\subseteq \CC^-$.

\section{Preliminaries}\label{sec:preliminaries}

In this paper, we consider linear \portHamiltonian systems without direct feedthrough, which are special cases of linear time-invariant systems.
General linear time-invariant systems take the form
\begin{align*}
    \dot x & = Ax + Bu,\\
    y & = C x + D u,
\end{align*}
where $A\in\RR^{n,n}, B\in \RR^{n,m}, C \in \RR^{p,n}$, and $D \in \RR^{p,m}$.
We will abbreviate the system as $(A,B,C,D)$ and write $(A,B,C)$ if $D=0$.
Throughout this paper, we assume $m,p \leq n$ and call these systems \emph{square} if $m=p$.
A linear system $(A,B,C,D)$ is termed \emph{minimal} if the pair $(A,B)$ is controllable and the pair $(A,C)$ is observable.

Linear \portHamiltonian (\pH) systems are systems of the form
\begin{align*}
    \dot x & = (J-R)Q x + (B-P) u,\\
    y & = (B+P)\tp Q x  + D u,
\end{align*}
where $B\in\RR^{n,m}$ and $J,R,Q\in\RR^{n,n}$ are such that
\begin{itemize}
   \item $J$ is skew-symmetric,
   \item $Q$ is symmetric positive definite, and
   \item $D,P$ and $R=R\tp$ satisfy
   \begin{equation*}
       \begin{bmatrix} R & P \\ P\tp & \tfrac12 (D+D\tp) \end{bmatrix} \succeq 0.
   \end{equation*}
\end{itemize}
We will be interested in the case $D = 0$ and $P=0$.
In this case we will write $(J,R,Q,B)$ to characterize the system and use the abbreviations $A:=(J-R)Q$ and $C:=B\tp Q$.
The function $\mathcal{H}: x \mapsto \tfrac12 x\tp Q x$ is usually called the \emph{Hamiltonian} of $(J,R,Q,B)$.
We will also call $Q$ the Hamiltonian.
Note that by our definition, all \pH systems are square.

To prove our main results, we make use of a preliminary result (\Cref{thm:interconnection-posreal}) which is concerned with the asymptotic stability of the closed loop system matrix of a system resulting from power-conserving interconnection of two systems which satisfy Lur'e equations related to the notions of (strict) positive realness.
Although the definitions of positive real systems are consistent in the literature, for strict positive realness there exist multiple definitions.
For example,~\cite{Hak20} distinguishes between ``weak strict positive realness'' and ``strict positive realness'', whereas in~\cite{And68} this distinction is not made.
An overview of different definitions and their connections is given in~\cite{Wen88}.
There it was argued that, due to their importance for stability analysis, the Lur'e equations should be used for the definition of strict positive realness.
We follow this argumentation, but emphasize the distinction from the existing definitions by avoiding the term ``positive real''.
Solutions of Lur'e equations were studied in, e.g.,~\cite{Rei11}.

We say that a square linear system $(A,B,C,D)$ satisfies
\begin{itemize}
    \item \emph{the Lur'e equations}, if $A$ is stable and there exists a symmetric positive definite matrix $P\in\RR^{n,n}$ and matrices $L \in \RR^{n,m}, W \in \RR^{m,m}$ that satisfy
    \begin{equation}\label{eq:lure-pr-D}
        \begin{aligned}
            A\tp P + P A & = - L L\tp \\
            C\tp - PB & = L W \\
            W\tp W & = D + D\tp.
        \end{aligned}
    \end{equation}
    \item \emph{the strong Lur'e equations}, if $A$ is asymptotically stable, $B$ has full column rank and there exist symmetric positive definite matrices $P,S\in\RR^{n,n}$ and matrices $L \in \RR^{n,m}, W \in \RR^{m,m}$ that satisfy
    \begin{equation}\label{eq:lure-spr-D}
        \begin{aligned}
            A\tp P + P A & = - L L\tp - S \\
            C\tp - PB & = L W \\
            W\tp W & = D + D\tp.
        \end{aligned}
    \end{equation}
\end{itemize}

Let us briefly point out a few differences to the standard notions of positive realness.
The usual definition, e.g.,~\cite{Wil72,BoyEFB94,And67} of positive real systems is in terms of the transfer function $G(s) = C(sI - A)\inv B + D$ of $(A,B,C,D)$, which has to be analytic in $\CC^+$ and satisfy
\begin{equation*}
    G(s) + G(s)\hm \succeq 0 ~\text{ for all }~  s \in \CC^+\cup i\RR.
\end{equation*}
It can then be shown that, under the assumption of minimality of $(A,B,C,D)$, positive realness is equivalent to the system satisfying~\eqref{lure-pr-D}~\cite{Wil71,BoyEFB94}.
For strict positive realness, there are multiple slightly different definitions in the literature.
For example, in~\cite{GuiO13} a transfer function is termed strictly positive real if it is analytic in $\CC^+$ and satisfies
\begin{equation*}
    G(s) + G(s)\hm \succeq \eta I ~\text{ for all }~  s \in \CC^+
\end{equation*}
for some $\eta > 0$.
On the other hand, in~\cite{And68} a transfer function $G$ that is termed strictly positive real satisfies
\begin{equation}\label{eq:spr-def-1-31}
    G(i\omega) + G(i\omega)\hm \succ 0 ~\text{ for all }~ \omega \in \RR.
\end{equation}
There, it is argued that for a strictly positive real transfer function $G$ and a sufficiently small $\varepsilon>0$ also $G(\cdot - \varepsilon)$ is positive real.
This may contradict the claims made in~\cite[Examples 3 and 4]{Hak20}, where~\eqref{spr-def-1-31} is the definition of ``weak strict positive realness''.

\begin{remark}\label{rem:kyp}
    We may reformulate~\eqref{lure-pr-D} as the matrix inequality
    \begin{equation}\label{eq:kyp-lmi-D}
        \begin{bmatrix} -PA - A\tp P & C\tp - PB \\ C - B\tp P & D + D\tp \end{bmatrix} = \begin{bmatrix} LL\tp & L W \\ W\tp L\tp & W\tp W \end{bmatrix} = \begin{bmatrix} L \\ W\tp \end{bmatrix} \begin{bmatrix} L\tp & W \end{bmatrix} \succeq 0, ~~ P = P\tp \succ 0.
    \end{equation}
    If $D=0$, then $W = 0$ and the above inequality simplifies to
    \begin{equation*}
        \begin{bmatrix} -PA - A\tp P & C\tp - PB \\ C - B\tp P & 0 \end{bmatrix} \succeq 0,
    \end{equation*}
    which is equivalent to
    \begin{equation*}
        PA + A\tp P \preceq 0 \und C = B\tp P.
    \end{equation*}

    We can proceed similarly for~\eqref{lure-spr-D} and observe that for a linear system $(A,B,C)$ where $A$ is asymptotically stable and $B$ has full column rank a sufficient condition for the satisfaction of the strong Lur'e equations is the existence of a symmetric positive definite matrix $P\in\RR^{n,n}$ such that
    \begin{equation*}
        P A + A\tp P \prec 0 \und C = B\tp P.
    \end{equation*}

    The inequality~\eqref{kyp-lmi-D} is linear in $P$ and is commonly known as \emph{Kálmán-Ya\-ku\-bo\-vich-Popov linear matrix inequality} (KYP-LMI).
    It was shown in~\cite{Wil71} that there exist extremal solutions $X_{\min}$ and $X_{\max}$ to~\eqref{kyp-lmi-D} in the sense that $0\preceq X_{\min} \preceq P \preceq X_{\max}$ for all solutions $P$ of~\eqref{kyp-lmi-D}.
    Similarly, there exist extremal solutions $Y_{\min}$ and $Y_{\max}$ to a dual KYP-LMI.
    Further, provided that $D+D\tp$ is nonsingular, the inequality~\eqref{kyp-lmi-D} can be associated with an algebraic Riccati equation using the Schur complement.
    The solutions $X_{\min}$ and $Y_{\min}$ are used for \emph{positive real balancing}~\cite{Obe91}.
\end{remark}

For the sake of completeness, let us recall some well-known observations concerning \pH systems, which are also stated in, e.g.,~\cite{BeaMV19}.
Suppose $(J,R,Q,B)$ is a \pH system and~$A$ and~$C$ are defined as usual.
Then $AQ\inv = J-R$, which shows that $J$ and $-R$ are the respective skew-symmetric and symmetric parts of $AQ\inv$.
Hence, we have
\begin{equation}\label{eq:JR-lure}
    J = \tfrac12 (AQ\inv - Q\inv A\tp) \und R = - \tfrac12 (AQ\inv + Q\inv A\tp).
\end{equation}
Since $R$ is symmetric positive semi-definite by assumption, also
\begin{equation*}
    0 \preceq -2 QRQ = - QA - A\tp Q.
\end{equation*}
As (by our definition) all \portHamiltonian systems are stable~\cite[Lemma 3.1]{MehMS16}, this shows that $(J,R,Q,B)$ satisfies the Lur'e equations.

With \Cref{rem:kyp} in mind, let us assume that for a general linear system $(A,B,C)$ there exists a symmetric positive definite matrix $P$ such that
\begin{equation*}
    PA + A\tp P \preceq 0 \und C = B\tp P.
\end{equation*}
Then by Sylvester’s law of inertia for the symmetric part of $AP\inv$ we have
\begin{equation*}
    P\inv (PA + A\tp P) P\inv = AP\inv + P\inv A\tp \preceq 0.
\end{equation*}
Since $C = B\tp P$ and $P$ is symmetric positive definite, we can choose $Q := P$ and define~$J$ and~$R$ as in~\eqref{JR-lure} to arrive at a \portHamiltonian formulation $(J,R,Q,B)$ of $(A,B,C)$.
We will use this construction in the proofs of \Cref{thm:mhinf,thm:mhinf-v2}.
Note that $R$ is definite and the system satisfies the strong Lur'e equations if $PA + A\tp P \prec 0$ and $B$ has full column rank.
For that, recall \Cref{rem:kyp} and see~\cite[Lemma 3.1]{MehMS16} for the asymptotic stability of~$A$.

We have seen that any positive definite solution $X$ to
\begin{equation}\label{eq:switchQ}
    \begin{bmatrix} -A\tp X - XA & C\tp - XB \\ C - B\tp X & 0 \end{bmatrix} \succeq 0, ~~ X = X\tp \succeq 0
\end{equation}
defines a \portHamiltonian representation of $(A,B,C)$ with $X$ as the Hamiltonian.
Since multiple solutions to~\eqref{switchQ} may exist, this shows that \pH formulations are not unique.
This fact is, for example, exploited in~\cite{BreMS21}, where the authors determine a \pH representation which leads to the minimization of an a priori error bound for model reduction.
Another example of this principle is~\cite{BeaMV19}, where maximally robust \pH representations are studied.
We will also make use of the non-uniqueness in \Cref{sec:numexp}.

Since there is a relationship between the Lur'e equations~\eqref{lure-pr-D} and~\eqref{lure-spr-D} and the possibility of a \portHamiltonian realization, we can ask how these properties influence a closed loop system resulting from power-conserving interconnection of two such systems.
\Cref{thm:interconnection-posreal} answers this question and appeared in similar form in~\cite{BenIJ81} and~\cite{HadB91}.

\begin{proposition}\label{thm:interconnection-posreal}
    Suppose $\Sigma_1 = (A_1,B_1,C_1)$ and $\Sigma_2 = (A_2,B_2,C_2)$ are systems that satisfy the Lur'e equations and can be coupled via power-conserving interconnection.
    If $\Sigma_1$ is minimal and $\Sigma_2$ satisfies the strong Lur'e equations, then
    \begin{equation*}
        \mathbf{A} := \begin{bmatrix} A_1 & - B_1 C_2 \\ B_2 C_1 & A_2 \end{bmatrix}
    \end{equation*}
    is asymptotically stable.
\end{proposition}
\begin{proof}
    Since $\Sigma_1$ and $\Sigma_2$ satisfy the (strong) Lur'e equations, there exist symmetric positive definite matrices $P_1$ and $P_2$, matrices $L_1$ and $L_2$, and a symmetric positive definite matrix $S_2$ of appropriate dimension such that
    \begin{align*}
        A_1\tp P_1 + P_1 A_1 + L_1 L_1\tp & = 0, \\
        B_1\tp P_1 - C_1 & = 0, \\
        A_2\tp P_2 + P_2 A_2 + L_2 L_2\tp + S_2 & = 0, \\
        B_2\tp P_2 - C_2 & = 0.
    \end{align*}
    To prove the asymptotic stability of $\mathbf{A}$, define $\mathbf{P}$ and $\mathbf{R}$ as
    \begin{equation*}
        \mathbf{P} := \begin{bmatrix} P_1 & 0 \\ 0 & P_2 \end{bmatrix}
        \und
        \mathbf{R} := \begin{bmatrix} R_1 & 0 \\ 0 & R_2 \end{bmatrix} := \begin{bmatrix} L_1 L_1\tp & 0 \\ 0 & L_2 L_2\tp + S_2 \end{bmatrix}.
    \end{equation*}
    and notice that they satisfy the Lyapunov equation
    \begin{equation}\label{eq:lyapunovAPR}
        \mathbf{A}\tp \mathbf{P} + \mathbf{P} \mathbf{A} + \mathbf{R} = 0.
    \end{equation}
    Now, assume $\mathbf{v}$ is a right eigenvector of $\mathbf{A}$ associated with the eigenvalue $\lambda \in \CC$.
    Multiplying~\eqref{lyapunovAPR} by $\mathbf{v}\hm$ from the left and $\mathbf{v}$ from the right, we obtain
    \begin{equation*}
        2 \Re(\lambda) \mathbf{v}\hm \mathbf{P} \mathbf{v} = - \mathbf{v}\hm \mathbf{R} \mathbf{v}.
    \end{equation*}
    Since $\mathbf{P} \succ 0$ and $\mathbf{R} \succeq 0$, we immediately see that $\Re(\lambda) \leq 0$.
    Now assume $\Re(\lambda) = 0$.
    Then also $\mathbf{R} \mathbf{v} = 0$.
    Partitioning $\mathbf{v} = \big[\begin{smallmatrix} v_1 \\ v_2\end{smallmatrix}\big]$ appropriately and noticing $\ker(R_2)=\{0\}$, we see that $v_2 = 0$.
    Rewriting $\mathbf{A} \mathbf{v} = \lambda \mathbf{v}$ as
    \begin{equation}\label{eq:blockEV}
        \begin{bmatrix} A_1 & - B_1 C_2 \\ B_2 C_1 & A_2 \end{bmatrix} \begin{bmatrix} v_1 \\ 0 \end{bmatrix} = \lambda \begin{bmatrix} v_1 \\ 0 \end{bmatrix},
    \end{equation}
    we notice that the second row reads
    \begin{equation}\label{eq:BCv}
        B_2 C_1 v_1 = 0.
    \end{equation}
    By assumption, the system $(A_2,B_2,C_2)$ satisfies the strong Lur'e equations.
    Hence, $B_2$ has full column rank and the matrix $B_2\tp B_2$ is invertible.
    Multiplying~\eqref{BCv} by $B_2\tp$ from the left, we obtain $C_1 v_1 = 0$.
    But the first row of~\eqref{blockEV} gives $A_1 v_1 = \lambda v_1$, which together with $C_1 v_1 = 0$ contradicts the assumed observability of $(A_1,C_1)$ using the Hautus test.
    This shows that $\Re(\lambda) = 0$ is not possible, so $\lambda \in \CC^-$ and $\mathbf{A}$ is asymptotically stable.
\end{proof}

\section{$\hinf$ Control and Structure Preserving Modifications} \label{sec:mainresult}
Let us state the $\hinf$ control problem that we will consider.
Assume a linear system under the influence of an additional control $w$ in the form of
\begin{align*}
    \dot{x} & = Ax + Bu + D_1 w, \\*
    y & = Cx + D_2 w
\end{align*}
is given, where $A \in \RR^{n,n}$, $B \in \RR^{n,m}, C \in \RR^{m,n}$ and $w$ has length $\ell \geq n+m$.
In the literature, $w$ is often assumed to be a Gaussian white noise process.
Here, we assume that $w$ is an additional deterministic control input and refer to, e.g.,~\cite{DulP13,Ste94} for the stochastic background.
We require the matrices $D_1\in\RR^{n,\ell}$ and $D_2\in\RR^{m,\ell}$ to satisfy $D_1 D_2\tp = 0$.
Consider the observed variable $z$ as
\begin{equation*}
    z = E_1 x + E_2 u,
\end{equation*}
where $E_1 \in \RR^{\ell,n}$ and $E_2\in\RR^{\ell,m}$ are such that $E_1\tp E_2 = 0$.

Now assume that the linear system
\begin{align*}
    \dot{\hat{x}} & = \hat{A} \hat{x} + \hat{B} \hat{u}, \\*
    \hat{y} & = \hat{C} \hat{x}
\end{align*}
is connected to the former system via power-conserving interconnection, i.e.~$\hat{u} = y$ and $u = - \hat{y}$.
Then the dynamics of the closed loop system are determined by
\begin{align*}
    \begin{bmatrix} \dot{x} \\ \dot{\hat x} \end{bmatrix}
    & =
    \begin{bmatrix} A & - B \hat C \\ \hat B C & \hat A \end{bmatrix}
    \begin{bmatrix} x \\ \hat x \end{bmatrix}
    +
    \begin{bmatrix} D_1 w\\ \hat B D_2 w \end{bmatrix}
\end{align*}
and $z$ takes the form
\begin{equation*}
    z = \begin{bmatrix} E_1 & - E_2 \hat C \end{bmatrix} \begin{bmatrix} x \\ \hat x \end{bmatrix}.
\end{equation*}
Note that the closed loop system matrix has the same form as $\mathbf{A}$ in \Cref{thm:interconnection-posreal}, which is why we denote
\begin{equation}\label{eq:hinf-ADE}
    \mathbf{A}:=\begin{bmatrix} A & - B \hat C \\ \hat B C & \hat A \end{bmatrix}, \quad
    \mathbf{D}:=\begin{bmatrix} D_1 \\ \hat B D_2 \end{bmatrix}, ~~ \text{and} ~~
    \mathbf{E}:=\begin{bmatrix} E_1 & - E_2 \hat C \end{bmatrix}.
\end{equation}
The closed loop transfer function $T_{zw}=T_{z\leftarrow w}$ from $w$ to $z$ is then given by
\begin{equation}\label{eq:Tzw-def}
    T_{zw} = \mathbf{E}(sI-\mathbf{A})\inv \mathbf{D}.
\end{equation}
The goal of $\hinf$ control design is to choose the controller $(\hat A, \hat B, \hat C)$ so that the closed loop system matrix $\mathbf{A}$ is asymptotically stable and the transfer function $T_{zw}$ satisfies
\begin{equation}\label{eq:Tzw-gamma}
    \norm{T_{zw}}_{\mathcal{H}_\infty} = \sup_{\omega \in \RR} \norm{T_{zw}(i\omega)}_2 < \gamma
\end{equation}
for some $\gamma \in (0,\infty)$.
Such a controller is termed \emph{admissible}.
Typically, the parameter $\gamma$ is not required to be optimal, i.e.~a smaller bound~$\tilde{\gamma}$ and a corresponding admissible controller might exist.
We denote the smallest value of $\gamma$ for which an admissible controller exists as~$\gamma_0$, only consider the case $\gamma > \gamma_0$, and call the corresponding controllers \emph{suboptimal}.
Suboptimal $\hinf$ controllers were extensively studied in~\cite{DoyGKF89}.

Unfortunately, even if the original system is \portHamiltonian, this is not necessarily also the case for the $\hinf$ controller stated in~\cite{DoyGKF89}.
This is demonstrated by \Cref{exmp:cls-hinf-not-pH}.

\begin{remark}
    \label{rem:cls-hinf}
    The $\hinf$ controller stated in~\cite{DoyGKF89} (which we will refer to as the \emph{classical $\hinf$ controller} in the following) is, in the simplest setting, constructed as follows.
    Suppose the stabilizing solutions $X$ and $Y$ of the algebraic Riccati equations
    \begin{equation}\label{eq:hinf-cls-c}
        A\tp X + X A - (1-\gamma^{-2}) X B B\tp X + C\tp C = 0
    \end{equation}
    and
    \begin{equation}\label{eq:hinf-cls-f}
        A Y + Y A\tp - (1-\gamma^{-2}) Y C\tp C Y + B B\tp = 0
    \end{equation}
    exist and that $\rho(XY) < \gamma^2$.
    Define $Z:=(I-\gamma^{-2}YX)\inv$ and
    \begin{equation*}
        \hat{A} := A - (1-\gamma^{-2}) Y C\tp C - BB\tp X Z, ~ \hat{B} := Y C\tp, ~ \hat{C} := B\tp X Z.
    \end{equation*}
    Then $(\hat{A},\hat{B},\hat{C})$ yields $\norm{T_{zw}}_\hinf < \gamma$.

    With our problem formulation, these AREs follow from the results of~\cite{DoyGKF89} when the matrices $D_1,D_2,E_1$ and $E_2$ are chosen as
    \begin{equation*}
        D_1 = \begin{bmatrix} B & 0 \end{bmatrix},~~ D_2 = \begin{bmatrix} 0 & I \end{bmatrix},~~ E_1 = \begin{bmatrix} C \\ 0 \end{bmatrix} \und E_2 = \begin{bmatrix} 0 \\ I \end{bmatrix}.
    \end{equation*}
\end{remark}

\begin{example}\label{exmp:cls-hinf-not-pH}
    Let us consider the \portHamiltonian system $(J,R,Q,B)$ defined by
    \begin{equation*}
    J =
    \begin{bmatrix}
         0 &  1 \\
        -1 &  0
    \end{bmatrix},~~
    R = Q =
    \begin{bmatrix}
        1 & 0 \\
        0 & 1
    \end{bmatrix}, \und
    B =
    \begin{bmatrix}
        2 & 0 \\
        0 & 1
    \end{bmatrix}.
    \end{equation*}
    As usual, let us define $A := (J-R)Q$ and $C:= B\tp Q$ and assume that the classical $\hinf$ controller $(\hat{A},\hat{B},\hat{C})$ is constructed as above for $\gamma = 2$.
    Since $(A,B,C)$ is \portHamiltonian, we can rewrite the matrix $\hat{B}$ as
    \begin{equation*}
        \hat{B} = Y C\tp = Y Q B.
    \end{equation*}
    If $(\hat{A},\hat{B},\hat{C})$ was \portHamiltonian, then there would exist a symmetric positive definite matrix $S$ such that $\hat{C} = \hat{B}\tp S$.
    Inserting the above equation and the definition of $\hat{C}$, we see that $S$ has to satisfy
    \begin{equation*}
        \hat{B}\tp S = B\tp Q Y S = B\tp X Z = \hat{C},
    \end{equation*}
    which can be written as a system of linear equations $P S = F$.
    For $J,R,Q$ and $B$ defined as above, the matrices $P$ and $F$ are given by
    \begin{equation*}
        P \approx
        \begin{bmatrix}
        1.6940  & -0.1497 \\
       -0.0749  &  0.4800
        \end{bmatrix}
        \und
        F \approx
        \begin{bmatrix}
        2.0592  &  0.1736 \\
        0.0868  &  0.5093
        \end{bmatrix},
    \end{equation*}
    where we only show the first five relevant digits.
    In particular, the matrix $P$ is invertible and the unique solution $S$ is given by
    \begin{equation*}
        S \approx
        \begin{bmatrix}
        1.2488  &  0.1990 \\
        0.3756  &  1.0919
        \end{bmatrix},
    \end{equation*}
    which is clearly not symmetric.
    Hence $(\hat{A},\hat{B},\hat{C})$ can not be represented as a \portHamiltonian system.
\end{example}

Similar to the LQG case discussed in~\cite{BreMS21}, our approach is to alter the algebraic Riccati equations~\eqref{hinf-cls-c} and~\eqref{hinf-cls-f} to ensure that the constructed controller has \portHamiltonian structure.
To show that the modiﬁed Riccati equations indeed yield a controller that ensures the closed loop error bound \eqref{Tzw-gamma}, we will follow the general idea of~\cite{HadBW93} and use multiple results of~\cite{BerH89}.
In contrast to our problem formulation, \cite{BerH89} considers the case where the original system and the control system are not connected by power-conserving interconnection but rather interconnected in the form $\hat{u} = y$ and $u=\hat{y}$.
Here, we have adapted the results of~\cite{BerH89} to fit our setting.
Further, for simplicity we define
\begin{equation*}
    V_1 := D_1 D_1\tp, ~~ V_2 := D_2 D_2\tp, ~~ R_1 := E_1\tp E_1, \und R_2:= E_2\tp E_2
\end{equation*}
and assume that both $V_2$ and $R_2$ are positive definite.
Later, by an appropriate choice for these matrices, we will be able to use \Cref{thm:bernstein-5.6} to show our main results.
In this context, the key observation from \Cref{thm:bernstein-5.6} is the guaranteed bound $\norm{T_{zw}}_{\hinf} < \gamma$.

\begin{proposition}[{\cite[Proposition 5.6]{BerH89}}]\label{thm:bernstein-5.6}
    Suppose that $(A,B,C)$ is a linear system, $(A,B)$ is stabilizable, $(A,C)$ is detectable, $\gamma>0$ and that there exist solutions $X=X\tp \succ 0$ and $Y=Y\tp \succeq 0$ of the AREs
    \begin{equation*}
        A Y + Y A\tp + V_1 + \gamma^{-2} Y R_1 Y - Y C\tp V_2^{-1} C Y = 0 \\
    \end{equation*}
    and
    \begin{align*}
        (A+\gamma^{-2}YR_1)\tp X + X(A+\gamma^{-2}YR_1) + R_1 & \\ - XBR_2^{-1}B\tp X + \gamma^{-2} X Y C^T V_2^{-1} C Y X & = 0.
    \end{align*}
    Further, assume that
    \begin{equation*}
        A+\gamma^{-2} Y R_1 + (\gamma^{-2}YC\tp V_2\inv C Y - B R_2\inv B\tp)X
    \end{equation*}
    is asymptotically stable and that
    \begin{equation*}
        \Big(A+\gamma^{-2}Y R_1 + X\inv R_1,~ \gamma^{-1} [R_1 + X B R_2\inv B\tp X]^{1/2}\Big)
    \end{equation*}
    is observable.
    Define a control system $(\hat{A},\hat{B},\hat{C})$ via
    \begin{equation*}
        \hat{A} := A - Y C\tp V_2\inv C - B R_2\inv B\tp X + \gamma^{-2} Y R_1,~~ \hat{B} := Y C\tp V_2\inv,~~
        \hat{C} := R_2\inv B\tp X.
    \end{equation*}
    If $\mathbf{A}$ as in~\eqref{hinf-ADE} is asymptotically stable, then the closed loop transfer function $T_{zw}$ satisfies $\norm{T_{zw}}_{\mathcal{H}_\infty} < \gamma$.
\end{proposition}

To state our first result, we make some assumptions regarding the matrices $V_1,V_2,R_1$ and $R_2$.
We assume $V_2 = R_2 = I, ~ R_1 = C\tp C$ and $V_1 = 2R + (1-\gamma^{-2})BB\tp$.
Under these assumptions, the algebraic Riccati equations found in \Cref{thm:bernstein-5.6} become~\eqref{hinf-mod-f} and~\eqref{hinf-mod-c}.
The physical interpretation of the terms $D_1,D_2,E_1$ and $E_2$ is, at least partially, lost.
As we have mentioned earlier, a similar approach was taken in~\cite{HadBW93}.

\begin{theorem}[structure-preserving $\hinf$ control]\label{thm:mhinf}
    Suppose $(J,R,Q,B)$ is a minimal \portHamiltonian system, define $A:=(J-R)Q$ and $C:=B\tp Q$, and assume $\gamma > 1$.
    Let $\hat{Y}=\hat{Y}\tp \succeq 0$ and $\hat{X}=\hat{X}\tp \succeq 0$ be the respective stabilizing solutions of the modified $\hinf$ filter equation
    \begin{equation}\label{eq:hinf-mod-f}
        A \hat{Y} + \hat{Y} A\tp - (1-\gamma^{-2})\hat{Y} C\tp C \hat{Y} + (1-\gamma^{-2})B B\tp + 2R = 0
    \end{equation}
    and the modified $\hinf$ control equation
    \begin{equation}\label{eq:hinf-mod-c}
        (A+\gamma^{-2}\hat{Y}C\tp C)\tp \hat{X} + \hat{X} (A+\gamma^{-2}\hat{Y}C\tp C) - (1-\gamma^{-2})\hat{X} B B\tp \hat{X} + C\tp C = 0.
    \end{equation}
    Define a control system via
    \begin{equation*}
        \hat{A} := A - (1-\gamma^{-2}) \hat{Y} C\tp C - B B\tp \hat{X}, ~~ \hat{B} := \hatY C\tp, ~~ \hat{C} := B\tp \hatX.
    \end{equation*}
    Then $(\hat{A},\hat{B},\hat{C})$ is \portHamiltonian.
    If additionally $(\hat{A},\hat{B},\hat{C})$ satisfies the strong Lur'e equations, then the transfer function $T_{zw}$ of the closed loop system $(\mathbf{A},\mathbf{D},\mathbf{E})$ satisfies
    \begin{equation*}
        \norm{T_{zw}} < \gamma.
    \end{equation*}
\end{theorem}
\begin{proof}
    The proof is carried out in four steps.
    \begin{enumerate}[(i)]
        \item
            We show that $\hat{Y}$ is given by $\hat{Y}=Q\inv$ and that $\hat{X}$ is positive definite. \label{item:mod-hinf-i}
        \item
            We show that the control system is \portHamiltonian. \label{item:mod-hinf-ii}
        \item
            We show that
            \begin{equation*}
                A+\gamma^{-2} \hat{Y} C\tp C + (\gamma^{-2} \hat{Y} C\tp C \hat{Y} - BB\tp) \hat{X}
            \end{equation*}
            is asymptotically stable and that
            \begin{equation*}
                \big( A + \gamma^{-2} \hat{Y} C\tp C + \hat{X}\inv C\tp C, ~~ \gamma^{-1} [C\tp C + \hat{X} B B\tp \hat{X}]^{1/2} \big)
            \end{equation*}
            is observable. \label{item:mod-hinf-iii}
        \item
            Using \Cref{thm:interconnection-posreal}, we argue that $\mathbf{A}$ is asymptotically stable.
            Then, together with \ref{item:mod-hinf-iii}, \Cref{thm:bernstein-5.6} shows the claim.\label{item:mod-hinf-iv}
    \end{enumerate}

    Let us begin to show \ref{item:mod-hinf-i}.
    Due to the assumed minimality of $(A,B,C)$, it follows from $\gamma > 1$ and Hautus tests that $(A\tp,(1-\gamma^{-2})C\tp C)$ is stabilizable and $(A\tp, (1-\gamma^{-2})BB\tp + 2R)$ is detectable.
    Hence,~\eqref{hinf-mod-f} has a unique stabilizing solution.
    This solution is given by $\hat{Y} = Q\inv$, since $Q$ is symmetric positive definite and we have
    \begin{align*}
      & ~ A Q\inv + Q\inv A\tp - (1-\gamma^{-2})Q\inv C\tp CQ\inv + (1-\gamma^{-2})BB\tp  + 2R \\
    = &\,(J-R) + (-J-R) - (1-\gamma^{-2}) B B\tp + (1-\gamma^{-2}) B B\tp + 2R = 0.
    \end{align*}
    Concerning $\hat{X}$, notice that the stabilizability of $(A+\gamma^{-2} B B\tp Q, (1-\gamma^{-2})BB\tp)$ and the observability of $\big(A+\gamma^{-2}\hat{Y}C\tp C, ~ C\tp C\big)$ follow from the controllability of $(A,B)$ and the observability of $(A,C)$ using Hautus tests.
    Hence, the stabilizing solution $\hat{X}$ is symmetric positive definite.

    To show that $(\hat{A},\hat{B},\hat{C})$ is \portHamiltonian, we define
    \begin{equation*}
        \hat{J} := \tfrac12 (\hat{A}\hat{X}\inv - \hat{X}\inv \hat{A}\tp), ~ \hat{R} := - \tfrac12 (\hat{A}\hat{X}\inv + \hat{X}\inv \hat{A}\tp), \und \hat{Q}:= \hat{X}.
    \end{equation*}
    By definition we have $\hat{A} = (\hat{J}-\hat{R})\hat{Q}$ and $\hat{B} = B$, so $\hat{B}\tp \hat{X} = B\tp \hat{X} = \hat{C}$.
    Since $\hat{J}$ is skew-symmetric by definition and $\hat{Q}=\hat{X}$ is symmetric positive definite, it remains to show that $\hat{R}$ is positive semi-definite.
    Notice that
    \begin{align}
      & ~ \hat{A}\tp \hat{X} + \hat{X} \hat{A} \nonumber \\*
    = & ~\!(A\tp - (1-\gamma^{-2}) C\tp C \hat{Y} - \hat{X} B B\tp) \hat{X} + \hat{X} (A - (1-\gamma^{-2})\hat{Y} C\tp C - B B\tp \hat{X}) \nonumber \\
    = & ~\!(A\tp + \gamma^{-2} C\tp C \hat{Y}) \hat{X} + \hat{X} (A + \gamma^{-2} \hat{Y} C\tp C) - 2 \hat{X} B B\tp \hat{X} - C\tp C \hat{Y} \hat{X} - \hat{X} \hat{Y} C\tp C. \nonumber \\
    \intertext{By plugging in \eqref{hinf-mod-c} and using $\hat{Y}=Q\inv$, we obtain}\nonumber
      & ~ \hat{A}\tp \hat{X} + \hat{X} \hat{A} \nonumber \\*
    = & - \gamma^{-2} \hat{X} B B\tp \hat{X} - C\tp C - \hat{X} B B\tp \hat{X} - C\tp B\tp \hat{X} - \hat{X} B C \nonumber \\
    = & - (C+B\tp \hat{X})\tp (C+B\tp \hat{X}) - \gamma^{-2} \hat{X} B B\tp \hat{X} \preceq 0. \label{eq:mhinf-definiteness}
    \end{align}
    Using Sylvester's law of inertia and $\hat{X}\inv = \hat{X}\mtp$, we conclude
    \begin{equation*}
        - 2 \hat{R} = \hat{A}\hat{X}\inv + \hat{X}\inv \hat{A}\tp = \hat{X}\inv (\hat{A}\tp \hat{X} + \hat{X} \hat{A}) \hat{X}\inv \preceq 0,
    \end{equation*}
    so $\hat{R}$ is symmetric positive semi-definite and the control system is \portHamiltonian.

    Regarding \ref{item:mod-hinf-iii}, first notice that
    \begin{align*}
        & ~ A + \gamma^{-2} \hat{Y} C\tp C + (\gamma^{-2} \hat{Y} C\tp C \hat{Y} - B B\tp)\hat{X} \\*
    = & ~ A + \gamma^{-2} \hat{Y} C\tp C - (1-\gamma^{-2}) BB\tp \hat{X} =: A_1.
    \end{align*}
    To show that this matrix is asymptotically stable, we use a standard fact regarding the solutions of Lyapunov equations associated with observable systems.
    For that, first note that $(A_1, [C\tp C + (1-\gamma^{-2}) \hat{X} B B\tp \hat{X}]^{1/2})$ is observable if and only if $(A_1, C\tp C + (1-\gamma^{-2})\hat{X} B B\tp \hat{X})$ is observable.
    Again using Hautus tests and $\gamma > 1$, we see that the latter matrix pair is indeed observable.
    Now we may deduce that $A_1$ is asymptotically stable if the Lyapunov equation
    \begin{align*}
        (A + \gamma^{-2} \hat{Y} C\tp C - (1-\gamma^{-2})B B\tp \hat{X})\tp P + P (A+\gamma^{-2}  \hat{Y} C\tp C - (1-\gamma^{-2}) BB\tp \hat{X} ) \\* + ~\! C\tp C + (1-\gamma^{-2})\hat{X} B B\tp \hat{X} = 0
    \end{align*}
    has a solution $P = P\tp \succ 0$.
    In fact, this solution is $P = \hat{X}$, since we may rearrange~\eqref{hinf-mod-c} as
    \begin{equation*}
        (A+\gamma^{-2} \hat{Y} C\tp C)\tp \hat{X} + \hat{X} (A+\gamma^{-2}\hat{Y} C\tp C) - (1-\gamma^{-2}) \hat{X} B B\tp \hat{X} = - C\tp C.
    \end{equation*}
    It remains to show that $\big( A + \gamma^{-2} \hat{Y} C\tp C + \hat{X}\inv C\tp C, ~ \gamma^{-1} [C\tp C + \hat{X} B B\tp \hat{X}]^{1/2} \big)$ is observable.
    Again, this is equivalent to the observability of
    \begin{equation*}
        \big( A + \gamma^{-2} \hat{Y} C\tp C + \hat{X}\inv C\tp C, ~ \gamma^{-2} (C\tp C + \hat{X} B B\tp \hat{X}) \big),
    \end{equation*}
    which follows from the observability of $(A,C)$.

    To show \ref{item:mod-hinf-iv}, first note that both $(A,B,C)$ and $(\hat{A},\hat{B},\hat{C})$ satisfy the Lur'e equations.
    If $(\hat{A},\hat{B},\hat{C})$ satisfies the strong Lur'e equations, then the asymptotic stability of $\mathbf{A}$ follows from \Cref{thm:interconnection-posreal}.
\end{proof}

As we have already mentioned, our approach is based on the results of~\cite{HadBW93}.
Let us remark some key differences of the two approaches.

\begin{remark}
    In~\cite{HadBW93}, the assumption $R_1 \succ C\tp R_2\inv C$ is made, which guarantees that the transfer function of the controller $(\hat{A},\hat{B},\hat{C})$ is strictly positive real in the sense of~\eqref{spr-def-1-31}.
    Our choices for $R_1$ and $R_2$ were made to more closely resemble the unmodified $\hinf$ control and filter equations.
    As a consequence, the matrix inequality $R_1 \succ C\tp R_2\inv C$ becomes an equality and the transfer function of $(\hat{A},\hat{B},\hat{C})$ does not necessarily have to satisfy the strong Lur'e equations.
    Since \Cref{thm:interconnection-posreal} is used to show the asymptotic stability of the closed loop matrix $\mathbf{A}$, and in turn that $\norm{T_{zw}}_\hinf < \gamma$ holds true, we needed to assume that $(\hat{A},\hat{B},\hat{C})$ satisfies the strong Lur'e equations.
    This differs from the results of~\cite{BreMS21}, where only minimality of the \portHamiltonian system $(A,B,C)$ was assumed.
\end{remark}

\begin{remark}
    We also see that the technical assumption $\gamma > 1$ was required to ensure that $(1-\gamma^{-2}) BB\tp$ is positive semi-definite.
    This allowed us to show the asymptotic stability of $A_1 = A + \gamma^{-2} \hat{Y} C\tp C - (1-\gamma^{-1}) B B\tp \hat{X}$ via the observability of a surrogate system.
    As noted in~\cite{MusG91}, our assumption $\gamma > 1$ is not a severe restriction, since for minimal systems $\gamma_0 \leq 1$ is only possible when the system matrix $A$ is asymptotically stable and has Hankel norm less than one.
\end{remark}

\begin{remark}
    If we compare the classical and modified $\hinf$ filter equations, which read as
    \begin{equation*}
        A Y + Y A\tp - (1-\gamma^{-2}) Y C\tp C Y + BB\tp = 0
    \end{equation*}
    and
    \begin{equation*}
        A \hat{Y} + \hat{Y} A\tp - (1-\gamma^{-2}) \hat{Y} C\tp C \hat{Y} + (1-\gamma^{-2})BB\tp + 2R = 0,
    \end{equation*}
    respectively, we see that, similar to the LQG case discussed in~\cite{BreMS21}, the filter equation is modified by adding the term $-\gamma^{-2} B B\tp + 2R$.
    As we have seen in the proof \Cref{thm:mhinf}, this ensures that $\hat{Y} = Q\inv$, which is another similarity to the LQG case.

    Unlike the LQG case, the classical and modified $\hinf$ control equations differ. They read as
    \begin{equation*}
        A\tp X + XA - (1-\gamma^{-2}) X B B\tp X + C\tp C = 0
    \end{equation*}
    and
    \begin{equation*}
        (A + \gamma^{-2} \hat{Y} C\tp C)\tp \hat{X} + \hat{X} (A + \gamma^{-2} \hat{Y} C\tp C) - (1-\gamma^{-2}) \hat{X} B B\tp \hat{X} + C\tp C = 0.
    \end{equation*}
    Since the solution $\hat{Y} = Q\inv$ is known a~priori and $\hat{Y} C\tp = B$, the modified $\hinf$ filter and control equations remain decoupled.
\end{remark}

\begin{remark}
    Notice that similar to the limiting behavior of the classical $\hinf$ controller, which approaches the classical LQG controller as $\gamma \to \infty$, taking the limit $\gamma \to \infty$ recovers the structure-preserving LQG controller developed in~\cite{BreMS21}.
\end{remark}

In \Cref{thm:mhinf}, the assumption that $(\hat{A},\hat{B},\hat{C})$ satisfies the strong Lur'e equations is a significant loss of generality and not satisfactory.
In order to overcome this assumption, note that $(\hat{A},\hat{B},\hat{C})$ satisfies the strong Lur'e equations when the matrix in~\eqref{mhinf-definiteness} is definite and $\hat{B}$ has full column rank.
As we will see, the former is ensured if we choose a symmetric positive definite matrix $P\in\RR^{n,n}$ and replace $R_1 = C\tp C$ by $R_1 = C\tp C + P$.
In turn, we also need to alter $V_1$ to $V_1 = 2R + (1-\gamma^{-2}) B B\tp - \gamma^{-2} Q\inv P Q\inv$.
The resulting modifications of the AREs are stated in \Cref{thm:mhinf-v2}.
Note that in order to satisfy our assumption $V_1 = D_1 D_1\tp$, the matrix $V_1$ needs to be positive semi-definite.

\begin{theorem}[structure-preserving $\hinf$ control -- version with $P$]\label{thm:mhinf-v2}
    Suppose $(J,R,Q,B)$ is a minimal \portHamiltonian system and that $B$ has full column rank, define $A:=(J-R)Q$ and $C:=B\tp Q$, and assume $\gamma > 1$.
    Let the symmetric positive definite matrix $P \in \RR^{n,n}$ be chosen such that $2R + (1-\gamma^{-2}) B B\tp - \gamma^{-2} Q\inv P Q\inv$ is positive semi-definite.
    Then $\tilde{Y}=Q\inv$ is a symmetric positive definite solution to the modified $\hinf$ filter equation
    \begin{equation}\label{eq:hinf-mod-f2}
        \begin{aligned}
                A \tilde{Y} + \tilde{Y} A\tp + \tilde{Y} ((\gamma^{-2}-1)C\tp C + \gamma^{-2}P) \tilde{Y}  &
            \\ + (1-\gamma^{-2}) B B\tp + 2R - \gamma^{-2} Q\inv P Q\inv                        & = 0.
        \end{aligned}
    \end{equation}
    Assume that $\tilde{X}=\tilde{X}\tp \succ 0$ is a solution to the modified $\hinf$ control equation
    \begin{equation}\label{eq:hinf-mod-c2}
        \begin{aligned}
            (A + \gamma^{-2} \tilde{Y} (C\tp C + P))\tp \tilde{X} + \tilde{X} (A+ \gamma^{-2}\tilde{Y}(C\tp C + P)) & \\
            - (1-\gamma^{-2})\tilde{X} B B\tp \tilde{X} + C\tp C + P & = 0,
        \end{aligned}
    \end{equation}
    and define a control system via
    \begin{equation*}
        \tilde{A} := A - (1-\gamma^{-2})\tilde{Y} C\tp C - B B\tp \tilde{X} + \gamma^{-2}\tilde{Y} P,
        ~~ \tilde{B} := \tilde{Y} C\tp, ~~
        \tilde{C} := B\tp \tilde{X}.
    \end{equation*}
    Then $(\tilde{A},\tilde{B},\tilde{C})$ is \portHamiltonian and the transfer function $T_{zw}$ of the closed loop system $(\mathbf{A},\mathbf{D},\mathbf{E})$ satisfies
    \begin{equation*}
        \norm{T_{zw}}_\hinf < \gamma.
    \end{equation*}
\end{theorem}

\begin{proof}
    The proof can be carried out very similarly to the proof of \Cref{thm:mhinf}.
    To see that a solution to~\eqref{hinf-mod-f2} is $\tilde{Y} = Q\inv$, notice that
    \begin{align*}
        (J-R) + (-J - R) + (\gamma^{-2} - 1) BB\tp + \gamma^{-2} Q\inv P Q\inv & \\* + (1-\gamma^{-2}) BB\tp + 2R - \gamma^{-2} Q\inv P Q\inv & = 0.
    \end{align*}
    To show that $(\tilde{A},\tilde{B},\tilde{C})$ is \portHamiltonian, we proceed as in \Cref{thm:mhinf} and define
    \begin{equation*}
        \tilde{J} := \tfrac12 (\tilde{A} \tilde{X}\inv - \tilde{X}\inv \tilde{A}\tp), ~ \tilde{R} := -\tfrac12 (\tilde{A} \tilde{X}\inv + \tilde{X}\inv \tilde{A}\tp), \und
        \tilde{Q} := \tilde{X}.
    \end{equation*}
    Then $\tilde{B} = \tilde{Y} C\tp = B$ and $\tilde{C} = B\tp \tilde{X} = \tilde{B}\tp \tilde{X}$.
    As $\tilde{J}$ is skew-symmetric by definition and $\tilde{Q} = \tilde{X}$ is symmetric positive definite, it remains to show that $\tilde{R}$ is symmetric positive semi-definite.
    Notice that
    \begin{equation*}
      \tilde{A}\tp \tilde{X} + \tilde{X} \tilde{A} = - P - (C + B\tp \tilde{X})\tp (C + B\tp \tilde{X}) - \gamma^{-2} \tilde{X} BB\tp \tilde{X} \prec 0,
    \end{equation*}
    so $\tilde{R} \succ 0$.
    To see that $\tilde{A}$ is asymptotically stable, either see~\cite[Lemma 3.1]{MehMS16} or use the observability of $(\tilde{A},P)$ together with the Hautus test and the existence of positive definite solutions to Lyapunov equations associated with observable systems.
    Further, notice that $\sigma_{\min}(\tilde{B}) = \sigma_{\min}(B) > 0$ since $B$ has full column rank.
    In particular, $(\tilde{A},\tilde{B},\tilde{C})$ satisfies the strong Lur'e equations and the asymptotic stability of $\mathbf{A}$ follows from \Cref{thm:interconnection-posreal}.
    To use \Cref{thm:bernstein-5.6} and finish the proof, we need to show that
    \begin{equation*}
        A + \gamma^{-2} \tilde{Y} C\tp C + \gamma^{-2} \tilde{Y} P + (\gamma^{-2} \tilde{Y} C\tp C \tilde{Y} - B B\tp) \tilde{X}
    \end{equation*}
    is asymptotically stable and that
    \begin{equation}\label{eq:mhinf2-obsv}
        \big( A + \gamma^{-2} \tilde{Y} (C\tp C + P) + \tilde{X}\inv (C\tp C + P), ~ \gamma^{-1} [C\tp C + P + \tilde{X} BB\tp \tilde{X}]^{1/2}\big)
    \end{equation}
    is observable.
    The proof of both of these claims follows along the lines of \Cref{thm:mhinf} and is given only briefly here.
    Observe that
    \begin{align*}
        & ~ A + \gamma^{-2} \tilde{Y} C\tp C + \gamma^{-2} \tilde{Y} P + (\gamma^{-2} \tilde{Y} C\tp C \tilde{Y} - B B\tp) \tilde{X} \\
    = & ~ A + \gamma^{-2} \tilde{Y} (C\tp C + P) - (1-\gamma^{-2}) BB\tp \tilde{X} =: A_2.
    \end{align*}
    Again, note that $(A_2, C\tp C + (1-\gamma^{-2}) \tilde{X} BB\tp \tilde{X} + P)$ is observable, which allows us to conclude that $A_2$ is asymptotically stable, since $\tilde{X}$ solves the Lyapunov equation
    \begin{align*}
        (A + \gamma^{-2} \tilde{Y} (C\tp C + P) - (1-\gamma^{-2})B B\tp \tilde{X})\tp \tilde{X} & \\
    + ~ \tilde{X} (A+\gamma^{-2} \tilde{Y} (C\tp C + P) - (1-\gamma^{-2}) BB\tp \tilde{X} ) & \\
    + ~ C\tp C + (1-\gamma^{-2})\tilde{X} B B\tp \tilde{X} + P & = 0.
    \end{align*}
    The observability of~\eqref{mhinf2-obsv} can be seen with the Hautus test using the invertibility of $P$.
\end{proof}

Let us note that in \Cref{thm:mhinf} we were able to deduce that stabilizing solutions to the Riccati equations exist, whereas in \Cref{thm:mhinf-v2} we did not show that $\tilde{Y}=Q\inv$ is stabilizing and needed to assume that a solution $\tilde{X} = \tilde{X}\tp \succ 0$ to~\eqref{hinf-mod-c2} exists.
The former is because the quadratic term in~\eqref{hinf-mod-f2} is indefinite and hence the positive definite solution $\tilde{Y} = Q\inv$ does not necessarily have to be stabilizing.
Necessary and sufficient conditions for the indefinite case are discussed in, e.g.,~\cite{Che92}.
Regarding the existence of a suitable solution to~\eqref{hinf-mod-c2}, see~\Cref{rem:existence-tildeX}.

\begin{remark}\label{rem:existence-tildeX}
    We can ensure that a solution $\tilde{X}=\tilde{X}\tp \succ 0$ to~\eqref{hinf-mod-c2} exists if
    \begin{equation} \label{eq:obsv-tildeX}
        \big( A+ \gamma^{-2} \tilde{Y} (C\tp C+P), ~ C\tp C + P \big)
    \end{equation}
    is observable and
    \begin{equation} \label{eq:stab-tildeX}
        \big( A+ \gamma^{-2} \tilde{Y} (C\tp C+P), ~ (1-\gamma^{-2}) BB\tp \big)
    \end{equation}
    is stabilizable.
    The observability of~\eqref{obsv-tildeX} follows from the invertibility of~$P$.
    Regarding the stabilizability of~\eqref{stab-tildeX}, Hautus tests, the fact that $B$ is assumed to have full rank and $\gamma > 1$ reveal that the matrix pair is stabilizable if and only if
    \begin{equation*}
        (A+\gamma^{-2} \tilde{Y} P,B)
    \end{equation*}
    is stabilizable.

    As we will see in \Cref{sec:modelreduction}, taking $P$ as a multiple of the Hamiltonian $Q$ is quite natural.
    In this special case, where $P=\alpha Q$ with $\alpha \in [0,\infty)$, we have $A + \gamma^{-2} \tilde{Y} P = A + \gamma^{-2} \alpha I$.
    Since identity shifts do not change the rank of the Kálmán matrix, we can deduce the controllability of $(A+\gamma^{-2}\alpha I, B)$.
    Hence, in this case a solution $\tilde{X}=\tilde{X}\tp \succ 0$ to~\eqref{hinf-mod-c2} exists.
\end{remark}

\begin{remark}
    Let us point out that the results can easily be extended to co-energy variable formulations of \pH systems.
    Such formulations are obtained by defining the new variable $z:= Qx$ and the matrix $E:=Q\inv$ and writing the system $(J,R,Q,B)$ as
    \begin{align*}
        E z & = (J-R)z + B u, \\
        y & = B\tp z.
    \end{align*}
\end{remark}

\section{Applications to Model Reduction}
\label{sec:modelreduction}
In this section, we will show how the algebraic Riccati equations from \Cref{thm:mhinf} and \Cref{thm:mhinf-v2} can be used to develop a structure-preserving model reduction method.
Our method is based on system balancing, which utilizes a change of coordinates that simultaneously diagonalizes the solutions to a pair of Lyapunov or Riccati equations.
Then, certain parts of the balanced system and their corresponding states are truncated.
For a general overview on balancing-related methods for model reduction we refer to, e.g.,~\cite{Ant05,BreS21}.
\emph{Classical $\hinf$ balancing}, i.e.~balancing with respect to the classical $\hinf$ algebraic Riccati equations~\eqref{hinf-cls-c} and~\eqref{hinf-cls-f}, was extensively studied in~\cite{MusG91}.
However, the classical approach has a major drawback when it comes to the approximation of \pH systems: the \portHamiltonian structure is not preserved during the model reduction process.
In other words, even when the full order system is \portHamiltonian, the reduced order model constructed by the balancing approach will not necessarily be~\pH.
As we will see, if the algebraic Riccati equations found in \Cref{thm:mhinf} and \Cref{thm:mhinf-v2} are used for system balancing, then the resulting balanced truncation method will be structure-preserving.
Accordingly, we will call this procedure \emph{modified $\hinf$ balanced truncation}.
Unfortunately, a few difficulties arise during the study of the model reduction error, and it is unclear how an a~priori error bound in the fashion of~\cite{MusG91} can be stated.

Now, let us proceed to our modified $\hinf$ balanced truncation method.
The first step is to balance the Gramians in question, i.e.~to simultaneously diagonalize the solutions $\hat{Y}=\tilde{Y}=Q\inv$ and $\hat{X},\tilde{X}$ to the algebraic Riccati equations from \Cref{thm:mhinf} and \Cref{thm:mhinf-v2}.
Here, we assume that the matrix pair $(A+\gamma^{-2}\tilde{Y}P,B)$ is stabilizable, which ensures that~\eqref{hinf-mod-c2} has a positive definite stabilizing solution $\tilde{X}$.
To check if standard balancing methods such as square root balancing are applicable, we examine if a state transformation $x\to Tx$ transforms the Gramians as $Y \to T Y T\tp$ and $X \to T\mtp \hat{X} T\inv$.
Here $Y\in \{ \hat{Y},\tilde{Y}\}$ and $X \in \{ \hat{X},\tilde{X} \}$, i.e.~this transformation property has to hold for both pairs of algebraic Riccati equations found in \Cref{thm:mhinf,thm:mhinf-v2}.
Let us focus on the second pair of algebraic Riccati equations, the pair~\eqref{hinf-mod-f2} and~\eqref{hinf-mod-c2}, which is sufficient since for $P=0$ the pair~\eqref{hinf-mod-f} and~\eqref{hinf-mod-c} is recovered.
Is is easy to see that a state transformation $x\to Tx$ transforms $R \to T R T\tp$ and $Q \to T\mtp Q T\inv$.
Let us assume that the state transformation yields $P \to \tilde{P}$.
If $\hat{Y}$ is a solution to
\begin{equation*}
    A \hat{Y} + \hat{Y} A\tp + \hat{Y} ((\gamma^{-2}-1)C\tp C + \gamma^{-2}P) \hat{Y} + (1-\gamma^{-2}) BB\tp + 2R + \gamma^{-2} Q\inv P Q\inv = 0,
\end{equation*}
then we want $\bar{Y} = T \hat{Y} T\tp$ to solve
\begin{align*}
    T A T\inv \bar{Y} + \bar{Y} T\mtp A\tp T\tp + \bar{Y} ((\gamma^{-2}-1) T\mtp C\tp C T\inv + \gamma^{-2}\tilde{P}) \bar{Y} \\* + (1-\gamma^{-2}) T B B\tp T\tp + 2 T R T\tp + \gamma^{-2}T Q\inv T\tp \tilde{P} T Q\inv T\tp = 0.
\end{align*}
This is the case if and only if the state transformation acts on $P$ as
\begin{equation}\label{eq:Pcond}
    P \to T\mtp P T\inv = \tilde{P},
\end{equation}
which is satisfied for $P=\alpha Q$, where $\alpha \in [0,\infty)$,
or more generally $P=\alpha S$, where $S$ is a solution to the KYP-LMI
\begin{equation}\label{eq:kyp-P-S}
    \begin{bmatrix} - SA - A\tp S & C\tp - S B \\ C - B\tp S & 0 \end{bmatrix} \succeq 0, ~~ S = S\tp \succ 0.
\end{equation}
If~\eqref{Pcond} holds and $\hat{X}$ is a solution to
\begin{equation*}
    (A+\gamma^{-2} \hat{Y}(C\tp C + P))\tp \hat{X} + \hat{X} (A+\gamma^{-2}\hat{Y}(C\tp C + P)) - (1-\gamma^{-2})\hat{X} BB\tp \hat{X} + C\tp C + P = 0,
\end{equation*}
then the state transformation yields
\begin{align*}
    (T A T\inv + \gamma^{-2} T \hat{Y} (C\tp C + P) T\inv)\tp \bar{X} + \bar{X} (T A T\inv + \gamma^{-2} T \hat{Y} (C\tp C + P) T\inv) \\* - (1-\gamma^{-2})\bar{X} T B B\tp T\tp \bar{X} + T\mtp C\tp C T\inv + T\mtp P T\inv = 0
\end{align*}
which is solved by $\bar{X} = T\mtp \hat{X} T\inv$.
Hence,~\eqref{Pcond} is sufficient to ensure that balancing is possible.

In~\cite{PolV12} it was shown that a partitioning
\begin{equation}\label{eq:ph-partition}
    J = \begin{bmatrix} J_{11} & J_{12} \\ J_{21} & J_{22} \end{bmatrix}, ~
    R = \begin{bmatrix} R_{11} & R_{12} \\ R_{21} & R_{22} \end{bmatrix}, ~
    Q = \begin{bmatrix} Q_{11} & Q_{12} \\ Q_{21} & Q_{22} \end{bmatrix}, \und
    B = \begin{bmatrix} B_1 \\ B_2 \end{bmatrix}
\end{equation}
leads to the system
\begin{equation*}
    (J_{11},R_{11},Q_{11}-Q_{12} Q_{22}\inv Q_{21},B_1)
\end{equation*}
being \portHamiltonian.
Using this result we can now show that our modified $\hinf$ balanced truncation method preserves \portHamiltonian structure.
\begin{theorem}\label{thm:mhinf-preserves-ph}
    Suppose $(J,R,Q,B)$ is a minimal \portHamiltonian system and $P = \eta S$, where $\eta \in [0,\infty)$ and $S$ solves the KYP-LMI~\eqref{kyp-P-S}.
    If $(J,R,Q,B)$ is balanced with respect to~\eqref{hinf-mod-f2} and~\eqref{hinf-mod-c2}, and partitioned as in~\eqref{ph-partition}, then the system
    \begin{equation*}
        (J_{11},R_{11},Q_{11},B_{1})
    \end{equation*}
    is \portHamiltonian.
    In particular, truncation of $((J-R)Q,B,B\tp Q)$ will preserve the \portHamiltonian structure.
\end{theorem}
\begin{proof}
    Since the solution $\hat{Y}$ of~\eqref{hinf-mod-f2} is given by $\hat{Y} = Q\inv$, in balanced coordinates with respect to~\eqref{hinf-mod-f2} and~\eqref{hinf-mod-c2} the matrix $\hat{Y}\inv = Q$ is diagonal.
    In particular, the off-diagonal blocks $Q_{12}$ and $Q_{21}$ are zero, which implies
    \begin{equation*}
        Q_{11}-Q_{12} Q_{22}\inv Q_{21} = Q_{11}.
    \end{equation*}
    The claim then immediately follows from the previously mentioned results of~\cite{PolV12}.
\end{proof}

Note that to state \Cref{thm:mhinf-preserves-ph} we do not need to assume that $B$ has full column rank.
This assumption is only needed to prove the error bound $\norm{T_{zw}}_\hinf < \gamma$, which is not required for \Cref{thm:mhinf-preserves-ph}.

For most balanced truncation methods, there exist a priori error bounds for the model reduction error.
For the procedure presented here, such an error bound is not easily established using the standard methods.
The next remarks highlight the main difficulties in that regard.

\begin{remark}\label{rem:a-priori-error-bound}
    To state an a priori error bound for classical $\hinf$ balanced truncation, in~\cite{MusG91} the authors used explicit constructions of coprime factorizations of the transfer function.
    The coprime factors are constructed using the results of~\cite{MeyF87} and the algebraic Riccati equation~\eqref{hinf-cls-f}.
    The authors could then establish a connection between the coprime factors and the characteristic values of the method.
    Unfortunately, this procedure is no longer viable for the pair of modified $\hinf$ equations~\eqref{hinf-mod-f} and~\eqref{hinf-mod-c} (and~\eqref{hinf-mod-f2} and~\eqref{hinf-mod-c2}).
    We will focus on the first pair and define $\beta := (1-\gamma^{-2})^{1/2}$ as in~\cite{MusG91}.

    Concerning~\eqref{hinf-mod-f}, with our definition of $\beta$ the equation may be rewritten as
    \begin{equation*}
        A\hat{Y} + \hat{Y}A\tp - \beta^2 \hat{Y}C\tp C \hat{Y} + \beta^2 BB\tp + 2R = 0.
    \end{equation*}
    Scaling by $\beta^2$ results in
    \begin{equation*}
        A(\beta^2 \hat{Y}) + (\beta^2 \hat{Y}) A\tp - (\beta^2 \hat{Y}) C\tp C (\beta^2 \hat{Y}) + \beta^2 ((\beta B)(\beta B)\tp + 2 R) = 0.
    \end{equation*}
    Hence, if $[(\beta B)(\beta B)\tp + 2 R]^{1/2} = L$, then $\bar{Y} := \beta^2 \hat{Y}$ solves
    \begin{equation*}
        A\bar{Y} + \bar{Y}A\tp - \bar{Y}C\tp C \bar{Y} + (\beta L)(\beta L)\tp = 0.
    \end{equation*}
    Using the results of~\cite{MeyF87}, we obtain a (left) coprime factorization of
    \begin{equation*}
        \beta \tilde{G} = C(sI - A)\inv (\beta L)
    \end{equation*}
    which is in general different from $\beta G$.

    Similarly, scaling~\eqref{hinf-mod-c} by $\beta^2$ yields
    \begin{equation*}
        (A+\gamma^{-2}\hat{Y} C\tp C)\tp (\beta^2 \hat{X}) + (\beta^2 \hat{X})(A+\gamma^{-2}\hat{Y} C\tp C) - (\beta^2 \hat{X}) BB\tp (\beta^2 \hat{X}) + (\beta C)\tp (\beta C) = 0,
    \end{equation*}
    which is solved by $\bar{X} := \beta^2 \hat{X}$.
    Using the results of~\cite{MeyF87} we can only state a (right) coprime factorization of
    \begin{equation*}
        \beta \bar{G} = (\beta C) (sI - A - \gamma^{-2}\hat{Y} C\tp C)\inv B,
    \end{equation*}
    which is, again, generally different from $\beta G$.
\end{remark}

\begin{remark}
    Another difficulty arises when the second pair of algebraic Riccati equations,~\eqref{hinf-mod-f2} and~\eqref{hinf-mod-c2}, is used.
    Since the quadratic term in~\eqref{hinf-mod-f2} is inherently indefinite, we can no longer guarantee that a stabilizing solution exists.
    Even if such a stabilizing solution exists, it is not immediately clear that $Q\inv$ is this solution.
    Since the results of~\cite{MeyF87} rely on stabilizing solutions to the Riccati equations, this approach can no longer be used to construct normalized coprime factorizations.
    If we assume that $\tilde{Y}=Q\inv$ is the stabilizing solution to~\eqref{hinf-mod-f2}, then a similar reasoning as in \Cref{rem:a-priori-error-bound} applies to the second pair of AREs as well.
\end{remark}

\section{Numerical Experiments} \label{sec:numexp}
In this section, we provide two numerical examples that naturally allow for a \portHamiltonian formulation.
We consider a mass-spring-damper system as in~\cite{GugPBV12}, and an example of a DC motor from~\cite{VanJ14}.
Our focus is not on demonstrating that the new methods outperform existing ones, but rather on illustrating the theory.
This is why we only consider systems of moderate state space dimension.
We compare the $\hinf$ control scheme from \Cref{thm:mhinf,thm:mhinf-v2} with the classical $\hinf$ control scheme from~\cite{DoyGKF89}.
Further, we compare the modified $\hinf$ balanced truncation scheme with classical $\hinf$ balanced truncation from~\cite{MusG91}.
Since in~\cite{BreMS21} it was shown that the extremal solutions to the associated KYP-LMI play a particularly important role for model reduction, we pay special attention to these solutions as Hamiltonians.

All simulations were obtained using \matlab R2021b (i64) in conjunction with Rosetta~2 on an Apple M1 Pro Processor with 8 cores and 16GB of unified memory.
Further, let us make the following remarks on the implementation.
\begin{itemize}
    \item
    As we have mentioned earlier, \Cref{thm:mhinf-v2} does not require the solutions to~\eqref{hinf-mod-f2} and~\eqref{hinf-mod-c2} to be stabilizing.
    Nevertheless, all considered algebraic Riccati equations were solved using the \matlab routine \texttt{icare}, which computes stabilizing solutions.
    The features of the \texttt{Control System Toolbox} were used to calculate the $\hinf$ errors.
    For the experiments involving the modified $\hinf$ balanced truncation method we chose $\gamma = 2$.

    \item
    The extremal solutions $X_{\min}$ and $X_{\max}$ to the KYP-LMI
    \begin{equation}\label{eq:numexp-kyp}
        \begin{bmatrix} - A\tp X - X A & C\tp - XB \\ C - B\tp X & 0 \end{bmatrix} \succeq 0, ~ X = X\tp \succeq 0
    \end{equation}
    are computed using a regularization approach, see~\cite{Wil72} and the discussion above Theorem 2 therein, and an artificial feedthrough term $D+D\tp = 10^{-12}I$.

    \item
    To ensure numeric stability, numerically minimal realizations of the systems were obtained by following the approach of~\cite{BreMS21}, for which the truncation tolerance was chosen as $\varepsilon_{\text{trunc}} = 10^{-12}$.
    Additionally, a sign convention for encountered singular value and $QR$ decompositions was enforced.
\end{itemize}

For details on the mass-spring-damper system, we refer to~\cite{GugPBV12}.
In the DC motor model of~\cite{VanJ14}, the system matrices $J,R,Q$ and $B$ are given by
\begin{equation*}
    J = \begin{bmatrix} 0 & -k \\ k & 0 \end{bmatrix}, ~~ R = \begin{bmatrix} r & 0 \\ 0 & b \end{bmatrix}, ~~ Q = \begin{bmatrix} 1/l & 0 \\ 0 & 1/j \end{bmatrix} \und B = \begin{bmatrix} 1 \\ 0 \end{bmatrix},
\end{equation*}
where $k$ is the gyrator constant, $r>0$ is associated with a resistor in the circuit, $b>0$ models friction in the motor, $l>0$ is the inductor constant, and $j>0$ models the rotational inertia of the motor.

\subsection{$\hinf$ Performance}
First, we compare the structure-preserving $\hinf$ controller to the classical $\hinf$ controller in terms of the achieved closed loop $\hinf$ norm $\norm{T_{zw}}_{\hinf}$.

For the classical setting, the matrices $D_1$ and $E_1$ were obtained as in \Cref{rem:cls-hinf}.
In the modified case, these matrices were constructed by calculating (potentially semi-definite) Cholesky-like factors of $V_1$ and $R_1$ and extending the calculated factors with zero padding.
In both cases, the matrices $D_2$ and $E_2$ were chosen as
\begin{equation*}
    D_2 = \begin{bmatrix} 0 & I \end{bmatrix} \und E_2 = \begin{bmatrix} 0 \\ I \end{bmatrix}.
\end{equation*}
For $\gamma$, we considered values from $[1.05,3.95]$.

In \Cref{fig:hinf-performance-msd}, the mass-spring-damper system from~\cite{GugPBV12} is considered.
We see that our modified $\hinf$ controller is outperformed by the classical controller, although both controllers yield a closed loop performance within the prescribed~$\hinf$ bound.
Here, only the controller constructed by \Cref{thm:mhinf} is considered, since with $P=\alpha Q$ the matrix $V_1 = 2R + (1-\gamma^{-2})BB\tp - \gamma^{-2}Q\inv P Q\inv$ from \Cref{thm:mhinf-v2} was indefinite even for small values of $\alpha$.

Similar observations can be made in \Cref{fig:hinf-performance-dc}, where we compare the controller synthesis methods for a model of a DC motor~\cite{VanJ14}.
The model constants were chosen as $k=1,\: r=2,\: b=1,\: l=1$ and $j=2$.
Here, our modified $\hinf$ controller with $P=0$ is again outperformed by the classical $\hinf$ controller.
Further, the controller from \Cref{thm:mhinf-v2} with $P=Q$ is added to the comparison, and yields the worst performance of all considered controllers, but still clearly stays within the prescribed $\hinf$ bound.

\subsection{Model Reduction}
Now, let us illustrate the theoretical results of \Cref{sec:modelreduction}.

In \Cref{fig:cmp_mhinf}, we show the results obtained for the balancing method when the algebraic Riccati equations~\eqref{hinf-mod-f} and~\eqref{hinf-mod-c} from \Cref{thm:mhinf} are used for balancing.
Besides the canonical realization of the \pH system, the results for the realizations associated with the extremal solutions of \eqref{numexp-kyp} are included.
For comparison, the classical $\hinf$ balanced truncation approach is also added.
As was already observed similarly in~\cite{BreMS21}, the error clearly depends on the chosen representation, and the maximal solution $X_{\max}$ corresponds to the smallest model reduction error.
In this case, the approximation quality is close to the approximation quality of the classical $\hinf$ balanced truncation approach.
When the minimal solution $X_{\min}$ is used to represent the system, the reduced order model fails to approximate the full-order system and the error stagnates.
The canonical representation of the system leads to much better performance, but is still clearly outperformed by the representation based on $X_{\max}$.

In \Cref{fig:cmp_mhinf_P}, the same system is considered, and the results for the balancing method with~\eqref{hinf-mod-f2} and~\eqref{hinf-mod-c2} from \Cref{thm:mhinf-v2} are shown.
Except for the representation associated with $X_{\min}$, which is ommited from the figure, the same parameters are considered.
The matrix $P$ is chosen as $P= 0.001Q$, where $Q$ stems from the canonical representation of the \pH system.
We observe that the error significantly increases, and that the maximal representation still outperforms the canonical representation.

\begin{figure}
    \centering
%
%
\definecolor{mycolor1}{rgb}{0.29020,0.48235,0.71765}%
\definecolor{mycolor2}{rgb}{0.99608,0.85490,0.54510}%
\definecolor{mycolor3}{rgb}{0.86667,0.23922,0.17647}%
\begin{tikzpicture}

\begin{axis}[%
width=.7\linewidth,
height=.45\linewidth,
at={(0in,0in)},
scale only axis,
xmin=1,
xmax=4,
xlabel style={font=\color{white!15!black}\small},
xlabel={$\gamma$},
ymin=0,
ymax=4,
ylabel style={font=\color{white!15!black}\small},
ylabel={$\|T_{zw}\|_{\mathcal{H}_\infty}$},
axis background/.style={fill=white},
title style={font=\bfseries\small},
title={},
legend style={legend cell align=left, align=left, draw=white!15!black, font=\footnotesize, at={(0.936,0.9)},anchor=north east}
]
\addplot [color=mycolor1, line width=1.4pt]
  table[row sep=crcr]{%
1.05	1.01373273953898\\
1.15	1.05889480267131\\
1.25	1.09163990891203\\
1.35	1.11625420890767\\
1.45	1.13528875917276\\
1.55	1.15034958049049\\
1.65	1.16249360946787\\
1.75	1.17244239198944\\
1.85	1.18070383813368\\
1.95	1.18764485728286\\
2.05	1.19353643670267\\
2.15	1.19858182318182\\
2.25	1.20293562739682\\
2.35	1.20672020909484\\
2.45	1.21003262930528\\
2.55	1.21295124352099\\
2.65	1.21554440538392\\
2.75	1.21784823037658\\
2.85	1.21989131310539\\
2.95	1.22173575318453\\
3.05	1.22339739041552\\
3.15	1.2248997385589\\
3.25	1.22626262986165\\
3.35	1.22750288858173\\
3.45	1.22863486469075\\
3.55	1.22967085989336\\
3.65	1.23062147013847\\
3.75	1.23149586287388\\
3.85	1.23230200292234\\
3.95	1.23304683761531\\
};
\addlegendentry{modified $\mathcal{H}_{\infty}$ approach ($P = 0$)}

\addplot [color=mycolor2, line width=1.4pt]
  table[row sep=crcr]{%
1.05	0.637624130571316\\
1.15	0.63910116281937\\
1.25	0.640140624493254\\
1.35	0.640904269610207\\
1.45	0.641484226727121\\
1.55	0.6419364511722\\
1.65	0.642296733927832\\
1.75	0.642588942726773\\
1.85	0.642829549071597\\
1.95	0.643030251090889\\
2.05	0.643199561949768\\
2.15	0.643343808449322\\
2.25	0.643467779252146\\
2.35	0.643575157503097\\
2.45	0.643668816649881\\
2.55	0.643751027100666\\
2.65	0.643823603369604\\
2.75	0.643888010650385\\
2.85	0.64394544319961\\
2.95	0.643996882794372\\
3.05	0.644043142884973\\
3.15	0.644084902331512\\
3.25	0.644122731457015\\
3.35	0.644157112364887\\
3.45	0.644188454927458\\
3.55	0.644217109473984\\
3.65	0.644243376938317\\
3.75	0.644267517034117\\
3.85	0.644289754885993\\
3.95	0.644310286442583\\
};
\addlegendentry{classical $\mathcal{H}_{\infty}$ approach}

\addplot [color=mycolor3, line width=1.4pt]
  table[row sep=crcr]{%
1.05	1.05\\
1.15	1.15\\
1.25	1.25\\
1.35	1.35\\
1.45	1.45\\
1.55	1.55\\
1.65	1.65\\
1.75	1.75\\
1.85	1.85\\
1.95	1.95\\
2.05	2.05\\
2.15	2.15\\
2.25	2.25\\
2.35	2.35\\
2.45	2.45\\
2.55	2.55\\
2.65	2.65\\
2.75	2.75\\
2.85	2.85\\
2.95	2.95\\
3.05	3.05\\
3.15	3.15\\
3.25	3.25\\
3.35	3.35\\
3.45	3.45\\
3.55	3.55\\
3.65	3.65\\
3.75	3.75\\
3.85	3.85\\
3.95	3.95\\
};
\addlegendentry{theoretical bound}

\end{axis}
\end{tikzpicture}%
    \caption{Comparison of the $\mathcal{H}_{\infty}$ performance of the modified $\mathcal{H}_{\infty}$ controller ($P\,=\,0$) and the classical $\hinf$ controller for a mass-spring-damper system of dimension $n\,=\,10$.}
    \label{fig:hinf-performance-msd}
\end{figure}
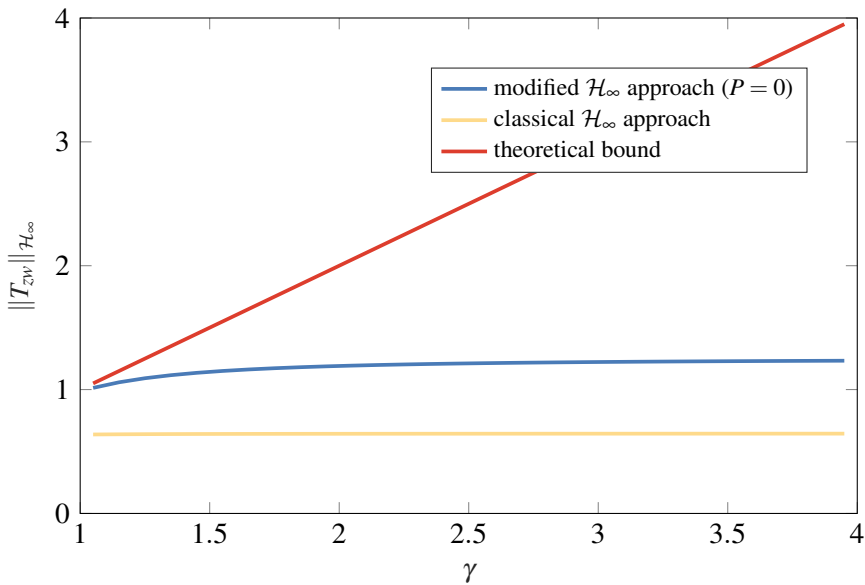

\begin{figure}
    \centering
%
%
\definecolor{mycolor1}{rgb}{0.29020,0.48235,0.71765}%
\definecolor{mycolor2}{rgb}{0.59608,0.79216,0.88235}%
\definecolor{mycolor3}{rgb}{0.99608,0.85490,0.54510}%
\definecolor{mycolor4}{rgb}{0.86667,0.23922,0.17647}%
\begin{tikzpicture}

\begin{axis}[%
width=.7\linewidth,
height=.45\linewidth,
at={(0in,0in)},
scale only axis,
xmin=1,
xmax=4,
xlabel style={font=\color{white!15!black}\small},
xlabel={$\gamma$},
ymin=0,
ymax=4,
ylabel style={font=\color{white!15!black}\small},
ylabel={$\|T_{zw}\|_{\mathcal{H}_\infty}$},
axis background/.style={fill=white},
title style={font=\bfseries\small},
title={},
legend style={legend cell align=left, align=left, draw=white!15!black, font=\footnotesize, at={(0.936,0.9)},anchor=north east}
]
\addplot [color=mycolor1, line width=1.4pt]
  table[row sep=crcr]{%
1.05	0.894566290669567\\
1.15	0.913742142946904\\
1.25	0.927967951879483\\
1.35	0.938859900356202\\
1.45	0.947407397482827\\
1.55	0.954250455762175\\
1.65	0.959820846255413\\
1.75	0.96441976609459\\
1.85	0.96826317986604\\
1.95	0.971509557614174\\
2.05	0.974277503749938\\
2.15	0.976657323833144\\
2.25	0.978718822475562\\
2.35	0.980516685112345\\
2.45	0.982094269297412\\
2.55	0.98348504694991\\
2.65	0.984718457314986\\
2.75	0.985817986767812\\
2.85	0.986802188199888\\
2.95	0.987686657702675\\
3.05	0.988484457245501\\
3.15	0.989206579097023\\
3.25	0.989862324369051\\
3.35	0.990459605362505\\
3.45	0.991005187256691\\
3.55	0.991504882269581\\
3.65	0.991963706559008\\
3.75	0.992386007779416\\
3.85	0.992775569391619\\
3.95	0.993135696442742\\
};
\addlegendentry{modified $\mathcal{H}_{\infty}$ approach ($P = 0$)}

\addplot [color=mycolor2, line width=1.4pt]
  table[row sep=crcr]{%
1.05	1.04995367173846\\
1.15	1.14713143399329\\
1.25	1.23412779310176\\
1.35	1.30659413608618\\
1.45	1.36474290625142\\
1.55	1.41102879031574\\
1.65	1.44817455608065\\
1.75	1.4783995251339\\
1.85	1.50334683351241\\
1.95	1.52420769988016\\
2.05	1.54185183728151\\
2.15	1.55692460186021\\
2.25	1.5699135373474\\
2.35	1.58119329034505\\
2.45	1.59105620883666\\
2.55	1.59973355288857\\
2.65	1.60741048371816\\
2.75	1.6142368559302\\
2.85	1.6203351232555\\
2.95	1.62580622100536\\
3.05	1.6307340050214\\
3.15	1.63518864399895\\
3.25	1.63922924175179\\
3.35	1.64290588534044\\
3.45	1.64626125995672\\
3.55	1.64933193327958\\
3.65	1.65214938512557\\
3.75	1.65474083900827\\
3.85	1.65712993832385\\
3.95	1.65933729970612\\
};
\addlegendentry{modified $\mathcal{H}_{\infty}$ approach ($P = Q$)}

\addplot [color=mycolor3, line width=1.4pt]
  table[row sep=crcr]{%
1.05	0.435075081728534\\
1.15	0.435185854846004\\
1.25	0.435267032614333\\
1.35	0.435328500714905\\
1.45	0.435376275921747\\
1.55	0.435414210605921\\
1.65	0.435444874055092\\
1.75	0.435470038388335\\
1.85	0.435490960846405\\
1.95	0.435508555070302\\
2.05	0.435523498910599\\
2.15	0.435536304455615\\
2.25	0.435547364776435\\
2.35	0.43555698588441\\
2.45	0.435565409045146\\
2.55	0.4355728266546\\
2.65	0.435579393726634\\
2.75	0.435585236332394\\
2.85	0.435590457886362\\
2.95	0.435595143887676\\
3.05	0.435599365537649\\
3.15	0.435603182529172\\
3.25	0.435606645218608\\
3.35	0.43560979633221\\
3.45	0.435612672318079\\
3.55	0.435615304425675\\
3.65	0.435617719574064\\
3.75	0.435619941055029\\
3.85	0.435621989106082\\
3.95	0.435623881380269\\
};
\addlegendentry{classical $\mathcal{H}_{\infty}$ approach}

\addplot [color=mycolor4, line width=1.4pt]
  table[row sep=crcr]{%
1.05	1.05\\
1.15	1.15\\
1.25	1.25\\
1.35	1.35\\
1.45	1.45\\
1.55	1.55\\
1.65	1.65\\
1.75	1.75\\
1.85	1.85\\
1.95	1.95\\
2.05	2.05\\
2.15	2.15\\
2.25	2.25\\
2.35	2.35\\
2.45	2.45\\
2.55	2.55\\
2.65	2.65\\
2.75	2.75\\
2.85	2.85\\
2.95	2.95\\
3.05	3.05\\
3.15	3.15\\
3.25	3.25\\
3.35	3.35\\
3.45	3.45\\
3.55	3.55\\
3.65	3.65\\
3.75	3.75\\
3.85	3.85\\
3.95	3.95\\
};
\addlegendentry{theoretical bound}

\end{axis}
\end{tikzpicture}%
    \caption{Comparison of the $\mathcal{H}_{\infty}$ performance of the modified $\mathcal{H}_{\infty}$ controller ($P\,=\,0$ and $P\,=\,Q$) and the classical $\hinf$ controller for a model of a DC motor.}
    \label{fig:hinf-performance-dc}
\end{figure}
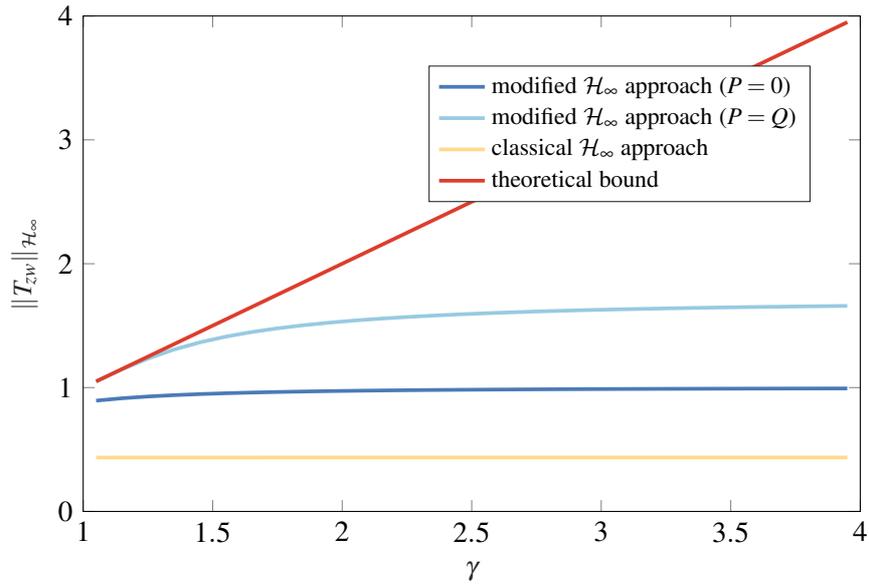

\begin{figure}[ht!] 
    \centering
%
%
\definecolor{mycolor1}{rgb}{0.29020,0.48235,0.71765}%
\definecolor{mycolor2}{rgb}{0.59608,0.79216,0.88235}%
\definecolor{mycolor3}{rgb}{0.99608,0.85490,0.54510}%
\definecolor{mycolor4}{rgb}{0.86667,0.23922,0.17647}%
\begin{tikzpicture}

\begin{semilogyaxis}[%
width=.7\linewidth,
height=.45\linewidth,
at={(0in,0in)},
scale only axis,
xmin=0,
xmax=40,
xlabel style={font=\color{white!15!black}\small},
xlabel={reduced system dimension $r$},
ymode=log,
ymin=8.49535387982952e-07,
ymax=1.45796462036564,
yminorticks=true,
ylabel style={font=\color{white!15!black}\small},
ylabel={$\|G-G_r\|_{\mathcal{H}_\infty}$},
axis background/.style={fill=white},
title style={font=\bfseries\small},
title={},
legend style={legend cell align=left, align=left, draw=white!15!black, font=\footnotesize, at={(0.936,0.9)},anchor=north east}
]
\addplot [color=mycolor1, line width=1.4pt]
  table[row sep=crcr]{%
2	0.738536244230906\\
4	0.326205055822555\\
6	0.276648284603498\\
8	0.18171753482425\\
10	0.132479005758916\\
12	0.0998343442338179\\
14	0.0744917576545375\\
16	0.054361340108815\\
18	0.0385759089272507\\
20	0.0269212496680331\\
22	0.0191841079416415\\
24	0.0137411795263631\\
26	0.00990095984302574\\
28	0.00705946379399129\\
30	0.00496699786947599\\
32	0.00345687011659769\\
34	0.00240400608583564\\
36	0.00168309311193405\\
38	0.00118415355853839\\
40	0.000831512516182473\\
};
\addlegendentry{modified $\mathcal{H}_\infty$ approach}

\addplot [color=mycolor2, line width=1.4pt]
  table[row sep=crcr]{%
2	1.45796462036564\\
4	1.42602485333101\\
6	1.38083956958252\\
8	1.32743011430961\\
10	1.26354235832982\\
12	1.20028397726424\\
14	1.13296346777127\\
16	1.13284533196086\\
18	1.13250638288909\\
20	1.09772310649379\\
22	1.07515072344695\\
24	1.06968930876809\\
26	1.06418982871952\\
28	1.03066573484071\\
30	1.00870807134063\\
32	1.00450039111315\\
34	0.961684649612224\\
36	0.941084757314632\\
38	0.935949656554745\\
40	0.893888482236671\\
};
\addlegendentry{modified $\mathcal{H}_\infty$ approach ($Q = X_{\mathrm{min}}$)}

\addplot [color=mycolor3, line width=1.4pt]
  table[row sep=crcr]{%
2	0.592433577144932\\
4	0.209859374576117\\
6	0.0936048531912181\\
8	0.0277029016517329\\
10	0.00868934242836131\\
12	0.00076648471019816\\
14	0.000226360645276524\\
16	3.8102282275558e-05\\
18	8.24430057520352e-06\\
20	8.06133277227016e-06\\
22	7.99127769913721e-06\\
24	7.72538485923877e-06\\
26	5.65945716799643e-06\\
28	6.02511145940424e-06\\
30	5.55760004043141e-06\\
32	5.49970594653645e-06\\
34	5.5329880416609e-06\\
36	5.52106480380101e-06\\
38	5.32859466263777e-06\\
40	5.36154161021165e-06\\
};
\addlegendentry{modified $\mathcal{H}_\infty$ approach ($Q = X_{\mathrm{max}}$)}

\addplot [color=mycolor4, line width=1.4pt]
  table[row sep=crcr]{%
2	0.366414869748339\\
4	0.153746811256731\\
6	0.0423951621783138\\
8	0.0134269925465447\\
10	0.00391416407730551\\
12	0.000380008987217835\\
14	0.000106707552257489\\
16	3.02024842825538e-05\\
18	2.91757752916411e-06\\
20	2.10802825058454e-06\\
22	2.59494065378672e-06\\
24	2.62347291255324e-06\\
26	2.19933051812333e-06\\
28	1.79822321108481e-06\\
30	1.6092451801376e-06\\
32	1.42264230808603e-06\\
34	1.25382636074091e-06\\
36	1.10349934846989e-06\\
38	9.6877336428096e-07\\
40	8.49535387982952e-07\\
};
\addlegendentry{classical $\mathcal{H}_\infty$ approach}

\end{semilogyaxis}
\end{tikzpicture}%
    \caption{Comparison of different choices of the Hamiltonian for MOR using modified $\hinf$ balanced truncation for a mass-spring-damper system of dimension $n=200$ (numerically minimal dimension $79$). We chose $P=0$ and added classical $\hinf$ balanced truncation for comparison.}
    \label{fig:cmp_mhinf}
\end{figure}
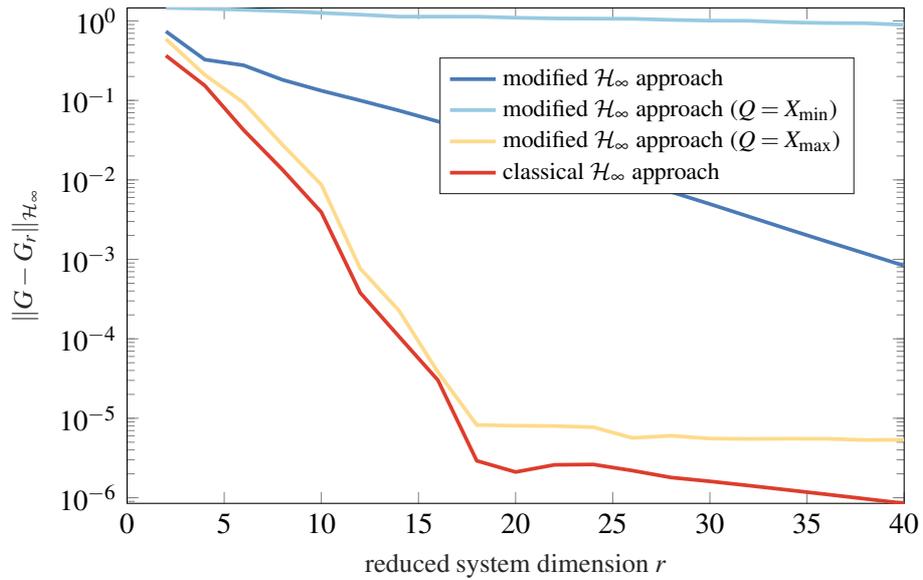

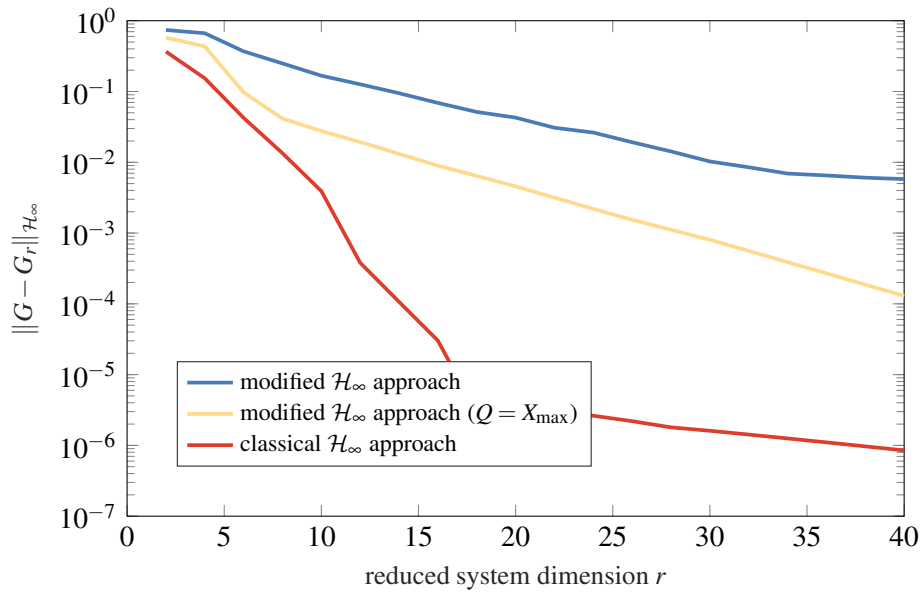
\begin{figure}[ht!] 
    \centering
%
%
\definecolor{mycolor1}{rgb}{0.29020,0.48235,0.71765}%
\definecolor{mycolor2}{rgb}{0.99608,0.85490,0.54510}%
\definecolor{mycolor3}{rgb}{0.86667,0.23922,0.17647}%
\begin{tikzpicture}

\begin{semilogyaxis}[%
width=.7\linewidth,
height=.45\linewidth,
at={(0in,0in)},
scale only axis,
xmin=0,
xmax=40,
xlabel style={font=\color{white!15!black}\small},
xlabel={reduced system dimension $r$},
ymode=log,
ymin=1e-07,
ymax=1,
yminorticks=true,
ylabel style={font=\color{white!15!black}\small},
ylabel={$\|G-G_r\|_{\mathcal{H}_\infty}$},
axis background/.style={fill=white},
title style={font=\bfseries\small},
title={},
legend style={legend cell align=left, align=left, draw=white!15!black, font=\footnotesize, at={(0.064,0.1)},anchor=south west}
]
\addplot [color=mycolor1, line width=1.4pt]
  table[row sep=crcr]{%
2	0.739802337098825\\
4	0.665617440433436\\
6	0.369250727869016\\
8	0.249157464471809\\
10	0.16722495147171\\
12	0.126365986510025\\
14	0.0945899156355544\\
16	0.0689905521350223\\
18	0.0514287665459673\\
20	0.0427378921152928\\
22	0.0307648825101209\\
24	0.0262812412003594\\
26	0.0191938054556274\\
28	0.0142269543247708\\
30	0.0102621844174226\\
32	0.00849976617497522\\
34	0.00692843679867739\\
36	0.006510260245412\\
38	0.00606019795921853\\
40	0.00578476017276581\\
};
\addlegendentry{modified $\mathcal{H}_\infty$ approach}

\addplot [color=mycolor2, line width=1.4pt]
  table[row sep=crcr]{%
2	0.576968548814318\\
4	0.433815296382372\\
6	0.0971606474264527\\
8	0.0414163895076008\\
10	0.0276881550957436\\
12	0.0193048878552972\\
14	0.0131659327468738\\
16	0.00892653488621224\\
18	0.00642565411871517\\
20	0.00456545847590797\\
22	0.00316687878237748\\
24	0.00219407697082828\\
26	0.00153867521222716\\
28	0.00111357085744315\\
30	0.000802557682972953\\
32	0.000559395864624678\\
34	0.000388832805508585\\
36	0.000269770662705493\\
38	0.000186706520818232\\
40	0.00012938170166151\\
};
\addlegendentry{modified $\mathcal{H}_\infty$ approach ($Q = X_{\mathrm{max}}$)}

\addplot [color=mycolor3, line width=1.4pt]
  table[row sep=crcr]{%
2	0.366414869748339\\
4	0.153746811256731\\
6	0.0423951621783138\\
8	0.0134269925465447\\
10	0.00391416407730551\\
12	0.000380008987217835\\
14	0.000106707552257489\\
16	3.02024842825538e-05\\
18	2.91757752916411e-06\\
20	2.10802825058454e-06\\
22	2.59494065378672e-06\\
24	2.62347291255324e-06\\
26	2.19933051812333e-06\\
28	1.79822321108481e-06\\
30	1.6092451801376e-06\\
32	1.42264230808603e-06\\
34	1.25382636074091e-06\\
36	1.10349934846989e-06\\
38	9.6877336428096e-07\\
40	8.49535387982952e-07\\
};
\addlegendentry{classical $\mathcal{H}_\infty$ approach}

\end{semilogyaxis}
\end{tikzpicture}%
    \caption{Comparison of different choices of the Hamiltonian for MOR using modified $\hinf$ balanced truncation for the mass-spring-damper system of dimension $n = 200$ (numerically minimal dimension $79$). We chose $P=0.001Q$. The minimal solution $X_{\min}$ was not considered.}
    \label{fig:cmp_mhinf_P}
\end{figure}

\section{Conclusion}\label{sec:conclusion}
In this paper, we propose a new method for the design of \portHamiltonian controllers with a guaranteed $\hinf$ bound for the closed loop transfer function.
To achieve this goal, we propose modifications of the algebraic Riccati equations used in classical $\hinf$ controller design.
Based on the modified algebraic Riccati equations, we additionally develop a structure preserving model reduction method.
Using numerical experiments, we illustrate that the approximation quality of the reduced order model depends on the chosen representation of the \portHamiltonian system, and that the maximal solution of the associated KYP-LMI appears to be best suited for the purpose of model reduction.

A natural first step for future research is to establish an a priori error bound for the presented $\hinf$ balanced truncation method.
With such an error bound, the effect of the choice of the matrix~$P$, the parameter $\gamma$ and the Hamiltonian $Q$ on the approximation quality can be studied in more detail.
Furthermore, to make the presented modified approaches feasible, future research should explore efficient and robust implementations, for example to compute the maximal solution~$X_{\max}$.

\subsection*{Acknowledgements} 
\vbox{We thank the Deutsche Forschungsgemeinschaft for their support within the project B03 in the Sonderforschungsbereich/Transregio 154 ``Mathematical Modelling, Simulation and Optimization using the Example of Gas Networks''.}

\bibliographystyle{siam}
\bibliography{references}

\begin{thebibliography}{10}

\bibitem{And67}
{\sc B.~Anderson}, {\em A system theory criterion for positive real matrices},
  SIAM Journal on Control, 5 (1967), pp.~171--182.

\bibitem{And68}
\leavevmode\vrule height 2pt depth -1.6pt width 23pt, {\em A simplified
  viewpoint of hyperstability}, IEEE Transactions on Automatic Control, 13
  (1968), pp.~292--294.

\bibitem{Ant05}
{\sc A.~Antoulas}, {\em Approximation of Large-Scale Dynamical Systems},
  Society for Industrial and Applied Mathematics, 2005.

\bibitem{BeaMV19}
{\sc C.~Beattie, V.~Mehrmann, and P.~{Van Dooren}}, {\em Robust
  {port-Hamiltonian} representations of passive systems}, Automatica, 100
  (2019), pp.~182--186.

\bibitem{BenIJ81}
{\sc R.~Benhabib, R.~Iwens, and R.~Jackson}, {\em Stability of large space
  structure control systems using positivity concepts}, Journal of Guidance
  Control and Dynamics, 4 (1981), pp.~487--494.

\bibitem{BerH89}
{\sc D.~Bernstein and W.~Haddad}, {\em {LQG} control with an
  $\mathcal{H}_{\infty}$ performance bound: A {Riccati} equation approach},
  IEEE Transactions on Automatic Control, 34 (1989), pp.~293--305.

\bibitem{BoyEFB94}
{\sc S.~Boyd, L.~{El Ghaoui}, E.~Feron, and V.~Balakrishnan}, {\em Linear
  Matrix Inequalities in System and Control Theory}, SIAM, 1994.

\bibitem{BreMS21}
{\sc T.~Breiten, R.~Morandin, and P.~Schulze}, {\em Error bounds for
  {port-Hamiltonian} model and controller reduction based on system balancing},
  Computers \& Mathematics with Applications,  (2021).

\bibitem{BreS21}
{\sc T.~Breiten and T.~Stykel}, {\em Balancing-Related Model Reduction
  Methods}, De Gruyter, 2021, pp.~15--56.

\bibitem{BreU21}
{\sc T.~Breiten and B.~Unger}, {\em Passivity preserving model reduction via
  spectral factorization}, arxiv preprint
  \href{https://arxiv.org/abs/2103.13194v3}{2103.13194v3}, 2021.

\bibitem{Che92}
{\sc S.~Chen}, {\em Necessary and sufficient conditions for the existence of
  positive solutions to algebraic {Riccati} equations with indefinite quadratic
  term}, Ap\-plied Mathematics and Optimization, 26 (1992), pp.~95--110.

\bibitem{Doy78}
{\sc J.~Doyle}, {\em Guaranteed margins for {LQG} regulators}, IEEE
  Transactions on Automatic Control, 23 (1978), pp.~756--757.

\bibitem{DoyGKF89}
{\sc J.~Doyle, K.~Glover, P.~Khargonekar, and B.~Francis}, {\em State-space
  solutions to standard $\mathcal{H}_2$ and $\mathcal{H}_\infty$ control
  problems}, IEEE Transactions on Automatic Control, 34 (1989), pp.~831--847.

\bibitem{DulP13}
{\sc G.~Dullerud and F.~Paganini}, {\em A Course in Robust Control Theory: A
  Convex Approach}, vol.~36, Springer, 2013.

\bibitem{GugPBV12}
{\sc S.~Gugercin, R.~Polyuga, C.~Beattie, and A.~{van der Schaft}}, {\em
  Structure-preserving tangential interpolation for model reduction of
  {port-Hamiltonian} systems}, Automatica, 48 (2012), pp.~1963 -- 1974.

\bibitem{GuiO13}
{\sc C.~Guiver and M.~Opmeer}, {\em Error bounds in the gap metric for
  dissipative balanced approximations}, Linear Algebra and its Applications,
  439 (2013), pp.~3659--3698.

\bibitem{HadB91}
{\sc W.~Haddad and D.~Bernstein}, {\em Explicit construction of quadratic
  {Lyapunov} functions for the small gain, positivity, circle and {Popov}
  theorems and their application to robust stability}, in Proceedings of the
  30th IEEE Conference on Decision and Control, vol.~3, 1991, pp.~2618--2623.

\bibitem{HadBW93}
{\sc W.~Haddad, D.~Bernstein, and Y.~Wang}, {\em Dissipative
  $\mathcal{H}_2$/$\mathcal{H}_{\infty}$ controller synthesis}, in American
  Control Conference, 1993, pp.~243--244.

\bibitem{Hak20}
{\sc M.~Hakimi-Moghadam}, {\em Positive real and strictly positive real {MIMO}
  systems: Theory and application}, International Journal of Dynamics and
  Control, 8 (2020), pp.~1--11.

\bibitem{KalB61}
{\sc R.~Kálmán and R.~Bucy}, {\em New results in linear filtering and
  prediction theory}, Journal of Basic Engineering, 83 (1961), pp.~95--108.

\bibitem{LozJ88}
{\sc R.~Lozano-Leal and S.~Joshi}, {\em On the design of the dissipative
  {LQG}-type controllers}, in Proceedings of the 27th IEEE Conference on
  Decision and Control, vol.~2, 1988, pp.~1645--1646.

\bibitem{MamZ22}
{\sc M.~Mamunuzzaman and H.~Zwart}, {\em Structure preserving model order
  reduction of port-{Hamiltonian} systems}, arXiv preprint
  \href{https://arxiv.org/abs/2203.07751v1}{2203.07751v1}, 2022.

\bibitem{MehMS16}
{\sc C.~Mehl, V.~Mehrmann, and P.~Sharma}, {\em Stability radii for linear
  {Hamiltonian} systems with dissipation under structure-preserving
  perturbations}, SIAM Journal on Matrix Analysis and Applications, 37 (2016),
  pp.~1625--1654.

\bibitem{MehU22}
{\sc V.~Mehrmann and B.~Unger}, {\em Control of {port-Hamiltonian}
  differential-al\-ge\-bra\-ic systems and applications}, arXiv preprint
  \href{https://arxiv.org/abs/2201.06590v1}{2201.06590v1}, 2022.

\bibitem{MeyF87}
{\sc D.~Meyer and G.~Franklin}, {\em A connection between normalized coprime
  factorizations and linear qua\-drat\-ic regulator theory}, IEEE Transactions
  on Automatic Control, 32 (1987), pp.~227--228.

\bibitem{MusG91}
{\sc D.~Mustafa and K.~Glover}, {\em Controller reduction by
  $\mathcal{H}_\infty$-balanced truncation}, IEEE Transactions on Automatic
  Control, 36 (1991), pp.~668--682.

\bibitem{Obe91}
{\sc R.~Ober}, {\em Balanced parametrization of classes of linear systems},
  SIAM Journal on Control and Optimization, 29 (1991), pp.~1251--1287.

\bibitem{PolV09}
{\sc R.~Polyuga and A.~{van der Schaft}}, {\em Moment matching for linear
  port-{Hamiltonian} systems}, in 2009 European Control Conference (ECC), 2009,
  pp.~4715--4720.

\bibitem{PolV10}
\leavevmode\vrule height 2pt depth -1.6pt width 23pt, {\em Structure preserving
  model reduction of port-{Hamiltonian} systems by moment matching at
  infinity}, Automatica, 46 (2010), pp.~665--672.

\bibitem{PolV12}
\leavevmode\vrule height 2pt depth -1.6pt width 23pt, {\em Effort- and
  flow-constraint reduction methods for structure preserving model reduction of
  {port-Hamiltonian} systems}, Systems \& Control Letters, 61 (2012),
  pp.~412--421.

\bibitem{Rei11}
{\sc T.~Reis}, {\em {Lur'e} equations and even matrix pencils}, Linear Algebra
  and its Applications, 434 (2011), pp.~152--173.

\bibitem{SchV20-2}
{\sc P.~Schwerdtner and M.~Voigt}, {\em {SOBMOR}: Structured optimization-based
  model order reduction}, arXiv preprint
  \href{https://arxiv.org/abs/2011.07567v2}{2011.07567v2}, 2020.

\bibitem{Ste94}
{\sc R.~Stengel}, {\em Optimal Control and Estimation}, Courier Corporation,
  1994.

\bibitem{Van00}
{\sc A.~{Van der Schaft}}, {\em $L^2$-Gain and Passivity Techniques in
  Nonlinear Control}, vol.~2, Springer, 2000.

\bibitem{Van06}
{\sc A.~{van der Schaft}}, {\em {Port-Hamiltonian} systems: An introductory
  survey}, in Proceedings of the International Congress of Mathematicians Vol.
  III, no.~suppl 2, European Mathematical Society Publishing House (EMS Ph),
  2006, pp.~1339--1365.

\bibitem{VanJ14}
{\sc A.~{Van der Schaft} and D.~Jeltsema}, {\em {Port-Hamiltonian} systems
  theory: An introductory overview}, Foundations and Trends in Systems and
  Control, 1 (2014), pp.~173--378.

\bibitem{Wen88}
{\sc J.~Wen}, {\em Time domain and frequency domain conditions for strict
  positive realness}, IEEE Transactions on Automatic Control, 33 (1988),
  pp.~988--992.

\bibitem{Wil71}
{\sc J.~Willems}, {\em Least squares stationary optimal control and the
  algebraic {Riccati} equation}, IEEE Transactions on Automatic Control, 16
  (1971), pp.~621--634.

\bibitem{Wil72}
\leavevmode\vrule height 2pt depth -1.6pt width 23pt, {\em Dissipative
  dynamical systems part {II}: Linear systems with qua\-drat\-ic supply rates},
  Archive for Rational Mechanics and Analysis, 45 (1972), pp.~352--393.

\end{thebibliography}

\end{document}